\documentclass{amsart}
\usepackage{amsmath,amssymb,bbm}
\usepackage{amsfonts,amssymb}

\usepackage{mathtools}
\usepackage{enumerate}

\usepackage{accents} 

\usepackage[numeric]{amsrefs}

\usepackage{epsfig}

\theoremstyle{plain} 

\newtheorem{thm}{Theorem}[section]
\newtheorem{cor}[thm]{Corollary}
\newtheorem{lem}[thm]{Lemma}
\newtheorem{prop}[thm]{Proposition}

\newtheorem{defn}[thm]{Definition}

\theoremstyle{remark}
\newtheorem{rem}[thm]{Remark}
\def\Ld{\Lambda}

\numberwithin{equation}{section}

\def\f{\frac}

\def\vi{\varphi}
\def\({\left(}
\def \){ \right)}

\def\Bl{\Bigl}
\def\Br{\Bigr}
 
 \def\ee{{\textnormal{e}}}
 
 \def\tr{{\triangle}}
 \def\Ga{\Gamma}

\def\ta{\theta}
\def\al{{\alpha}}
\def\da{{\delta}}
\def\sa{{\sigma}}
 
 \def\b{{\beta}}

 \def\k{{\kappa}}

 \def\va{\varepsilon}

 \def\nb{{\mathbf n}}

 \def\BB{{\mathbb B}}

 \def\NN{{\mathbb N}}

 \def\RR{{\mathbb R}}
 \def\SS{{\mathbb S}}

 \def\ZZ{{\mathbb Z}}

  \def\dim{\operatorname{dim}}

  \def\sph{\mathbb{S}^{d-1}}
  
\def\Og{\Omega}

\def\al{\alpha}
\newcommand{\wt}{\widetilde}
\newcommand{\wh}{\widehat}
\def\sub{\substack}

\def\p{\partial}
\def\ld{\lambda}
\def\bl{\bigl}
\def\br{\bigr}
\def\og{\omega}
\def\Ld{\Lambda}

\newcommand{\R}{{\mathbb{R}}}

\newcommand{\x}{{\boldsymbol{x}}}

\newcommand{\conv}{{\rm conv\,}}

\newcommand{\diam}{{\rm diam}}
\newcommand{\dist}{{\rm dist}}

\def\be{\begin{equation}}
\def\ee{\end{equation}}

 \def\osc{\operatorname{osc}}

\begin{document}

\title[]{$L^p$-Bernstein  inequalities  on $C^2$-domains and applications to discretization}
\author{Feng Dai}
\address{F.~Dai, Department of Mathematical and Statistical Sciences\\
	University of Alberta\\ Edmonton, Alberta T6G 2G1, Canada.}
\email{fdai@ualberta.ca}

\author{Andriy Prymak}
\address{Department of Mathematics, University of Manitoba, Winnipeg, MB, R3T2N2, Canada}

\email{prymak@gmail.com}

\thanks{	The first author was supported by  NSERC of Canada Discovery
	grant RGPIN-2020-03909, and the second author  was supported by NSERC of Canada Discovery grant RGPIN-2020-05357.
	}



\begin{abstract}
We prove a new Bernstein type inequality in $L^p$ spaces associated with the normal and the tangential derivatives on the boundary of a general compact $C^2$-domain. We give two applications: Marcinkiewicz type inequality for discretization of $L^p$ norms  and positive cubature formulas. Both results are optimal in the sense that the number of function samples used has the order of the dimension of the corresponding space of algebraic polynomials. 
\end{abstract}

\maketitle

%

\section{Introduction and Main Results}

We start with the two most classical inequalities for algebraic  polynomials: the Bernstein inequality 
$$ \|\vi^r P_n^{(r)}\|_{L^p[-1,1]} \leq C n^r \|P_n\|_{L^p[-1,1]},\  \ 0<p\leq \infty,\  \ r=1,2,\dots$$
and 
the Markov inequality 
$$ \|P_n^{(r)}\|_{L^p[-1,1]} \leq C n^{2r} \|P_n\|_{L^p[-1,1]},\  \ 0<p\leq \infty,\  \ r=1,2,\dots,$$
where $\vi(x)=\sqrt{1-x^2}$ for $x\in [-1,1]$,  $P_n$ is any algebraic polynomial of degree at most $n$, and   $C>0$ is a constant depending only on $r$ and $p$, see~\cite[Ch.~4]{De-Lo} and references therein. (Here and throughout the paper,  we do not care about the exact values of the constants.) These inequalities have been  playing   crucial roles in  proving various   results in  approximation theory, see, for instance, \cite{Di-To, De-Lo}.  As a result, they  have been generalized and improved in many  directions, see  \cite{O2, BE,Di-To,Er, Er2,  KNT, KL,MT2, O13, O14}. In particular,  they were generalized on  domains related to the spheres  for 
 $0<p\leq \infty$  in  \cite{Da06}, and   on more  general compact   domains  for $1\le p<\infty$ in \cite{Kr09} and for $p=\infty$  in  \cite{Ba,Du,Kr13,Kr-Re,To14,To17}.

  The aim of this paper is to establish $L^p$-Bernstein(-Markov) inequalities for polynomials on compact $C^2$-domains in $\RR^d$. 
 In this introduction,  we will describe our main results, the organization of the paper, the structure of the proofs and the methods used including a brief discussion of relevant earlier works.

 First,  we give  the   definition of $C^2$-domains.
 Let $B_r(\xi):=\{\eta\in\RR^d:\  \ \|\eta-\xi\|< r\}$
 denote  the open ball with center $\xi\in\RR^d$ and radius $r>0$   in $\R^{d}$, and 
 $B_r[\xi]$ the closure of $B_r(\xi)$.

 \begin{defn}\label{def-C2}
 	A bounded   set   $\Og$ in $ \RR^{d}$ is called a $C^2$-domain    if there exist a positive constant  $\k_0$, and a finite cover of the boundary $\p \Og$  by  open sets $\{ U_j\}_{j=1}^J$ in $\RR^d$ such that\begin{enumerate}[\rm (i)]
 		\item   for each $1\leq j\leq J$, there exists a function $\Phi_j\in C^2(\RR^d)$ such that 
 		$$U_j\cap \p\Og=\{ \xi\in U_j:\  \ \Phi_j(\xi)=0\}\   \ \text{and}\   \ \nabla \Phi_j(\xi)\neq 0,\   \   \    \   \forall \xi\in U_j\cap \p\Og; $$
 		\item for each $\xi\in \p\Og$, 
 		$$B_{\k_0} (\xi-\k_0 \mathbf n_\xi)\subset \Og\   \ \text{and}\     \   B_{\k_0} ( \xi+\k_0\mathbf  n_\xi) \subset \Og^c=\RR^d\setminus \Og,$$
 		where $\mathbf n_\xi$ denotes the unit outer normal vector to $\p\Og$ at $\xi$.
 	\end{enumerate}
 \end{defn}

Throughout this paper, unless otherwise stated, the greek letter   $\Og$ is always used to denote   a compact $C^2$-domain  with boundary $\Ga:=\p \Og$   in $\RR^d$.

 \begin{rem}
 	For the purposes of this paper, it was convenient to include the rolling ball property (ii) into the definition. However, under appropriate additional assumptions (e.g. if $\Omega$ and $\Omega^c$ have finitely many path-connected components) (ii) can be obtained as a consequence of (i), see~\cite[Th.~1(iii), (v), Rem.~3, p.~304]{Wa}.
 \end{rem}

 Next, we introduce some necessary notations for the rest of this paper.  Let $\sph:=\{x\in\RR^d: \|x\|=1\}$ denote  the unit sphere of $\RR^d$ with $d\ge 2$.  Here and throughout the paper,   $\|\cdot\|$ denotes the Euclidean norm of $\RR^d$. Given a set $E\subset \RR^d$, we denote by  $|E|$  the  Lebesgue measure of  $E$ in  $\RR^d$, and $\# E$ the cardinality of the set $E$. Furthermore, we define $\diam(E)=\max_{\xi,\eta\in E} \|\xi-\eta\|$, and  
 denote by  $\dist(\xi, E)$ the distance from a point $\xi\in\RR^d$ to a set  $E\subset \RR^d$; that is,  $\dist(\xi, E):=\inf_{\eta\in E}\|\xi-\eta\|$. 
   We denote by  
  $L^p(E)$, $0<p<\infty$  the usual Lebesgue $L^p$-space defined with respect to the $d$-dimensional Lebesgue measure $dx$ on the set $E\subset \RR^d$.  In the case when $p=\infty$ and $E\subset \RR^d$ is compact,  we set $L^\infty(E)=C(E)$, the space of all continuous functions  on $E$ with the uniform norm $\|\cdot \|_\infty$.

  Let $\Pi^d_n$ denote  the space of all real  algebraic polynomials in $d$ variables  of total degree at most $n$. Given $\xi\in\RR^d$, we denote by  $\p_\xi f=(\xi\cdot\nabla) f$ the directional derivative of $f$ along the direction of $\xi$, and  for a positive  integer $\ell$, we define
$\p_\xi^\ell:=(\xi\cdot \nabla)^\ell$. Throughout this paper,  the letter  $c$  denotes a generic constant whose value may
change from line to line, and the notation  $A \sim B$ means that  there exists a constant $c>0$, called the constant of equivalence, such that  $c^{-1} B\leq A \leq c B$.

The  $L^p$-Bernstein-Markov inequalities on $\Og$ are  formulated in terms of the normal and tangential derivatives prescribed on the boundary $\Ga$ of $\Og$. 
To be precise, 
we denote by $\mathbf{n}_\eta$  the  outer unit normal vector   to  $\Ga$ at  $\eta\in\Ga$.
For  $ \xi\in\Og$,  $ f\in C^\infty(\Og)$ and nonnegative integers $l_1, l_2$,
we define 
$$ \mathcal{D}_{\eta}^{l_1, l_2} f(\xi):= \max\Bl\{ 
\bl|  \p_{\pmb{\tau}}  ^{l_1} \p_{\nb_\eta}^{l_2} f(\xi)\br|:\   \   \pmb\tau\in\sph,\  \ \pmb\tau\cdot \mathbf n_\eta=0\Br\},\  \ \eta\in\Ga. $$
Furthermore, given  a parameter $\mu\ge \sqrt{\diam(\Og)}+1$, we define 
\begin{equation}\label{eqn:D max def}
	 \mathcal{D}_{n,\mu}^{l_1, l_2} f(\xi):=\max\Bl\{  \mathcal{D}_{\eta}^{l_1, l_2} f(\xi):\   \ \eta\in\p \Og,\   \  \|\eta-\xi\|\leq \mu \vi_{n,\Ga}(\xi)\Br\},\   \ \xi\in\Og,
\end{equation}
where 
\begin{equation*}
\vi_{n,\Ga}(\xi):=\sqrt{\dist(\xi, \p\Og)} +n^{-1},\   \ n=1,2,\dots, \xi\in\Og.
\end{equation*}
Here, the inclusion of the term $n^{-1}$ in the distance function $\vi_{n,\Ga}$ allows us  to combine the  Bernstein and Markov type inequalities on a single line.

With the above notation, we can state our main result  as follows:

\begin{thm}\label{thm-10-1-00} Let $\Og\subset \RR^d$  be   a compact $C^2$-domain  with boundary $\Ga:=\p \Og$, and let  $\mu\ge \sqrt{\diam(\Og)}+1$ be a given parameter. Then with the above notation, we have that   for any    $f\in \Pi_n^{d}$ and $0<p\leq \infty$, 
	\begin{equation}\label{Bernstein-tan-0}
		\Bl\|\vi_{n,\Ga}^j \mathcal{D}_{n,\mu}^{r, j+l} f \Br\|_{L^p(\Og)}\leq C n^{r+j+2l} \|f\|_{L^p(\Og)},\  \ r,j,l=0,1,\dots, 
	\end{equation}  
	where the constant $C>0$ depends only on $\Og$, $\mu$ and $p$ as $p\to 0$.
\end{thm}

\begin{rem}
	\begin{enumerate}[\rm (i)]
		\item Note that the inequality ~\eqref{Bernstein-tan-0} estimates not only the directional derivatives of $f$, but also  a maximum of such derivatives in a certain neighborhood defined in~\eqref{eqn:D max def}. Therefore, the result can be viewed as a maximal function type inequality.
		\item It can be easily seen  that the order $n^{r+j+2l}$ on the right hand side of \eqref{Bernstein-tan-0} is sharp as $n\to\infty$.
		Note  that  the standard one-dimensional Bernstein-Markov inequality applied along straight line segments from $\Og$ would only  give 
	\begin{equation*}
\Bl\|\vi_{n,\Ga}^j \mathcal{D}_{n,\mu}^{r, j+l} f \Br\|_{L^\infty(\Og)}\leq C n^{2r+j+2l} \|f\|_{L^\infty(\Og)},\  \ r,j,l=0,1,\dots.
\end{equation*}  
The improvement from $ n^{2r+j+2l} $  to $ n^{r+j+2l} $ for the full range of $1\leq p\leq \infty$  is exactly what is needed in many applications and reflects the correct order for the tangential component.

\item The same result is also true if we replace $\p_{\pmb \tau} ^{l_1}$ with mixed directional derivatives $\p_{\pmb\tau_1} ^{\al_1} \dots \p_{\pmb \tau_{d-1}}^{\al_{d-1}}$  in the definition of $\mathcal{D}_\eta ^{l_1, l_2}$, where $\{\pmb \tau_1, \dots, \pmb \tau_{d-1}\}$ is an orthonormal basis of the  space $\{ \pmb\tau\in\RR^d:\  \ \pmb\tau\cdot \mathbf n _\eta=0\}$ of tangent vectors to $\Ga$ at $\eta$,  and $\al_1+\dots+\al_{d-1}=l_1$. See Theorem ~\ref{thm:2d bern-2} for the details.

\item  Note that   the operators $\mathcal{D}^{l_1, l_2} _{n,\mu}$, $l_1, l_2=0,1,\dots$ are not commutative, which  means that  Theorem ~\ref{thm-10-1-00} for higher-order derivatives cannot be deduced from the corresponding result for lower order derivatives. 

		\item Computationally, the quantity $\dist(\xi, \p\Og)$ appearing in $\vi_{n,\Ga}$ may be inconvenient requiring certain non-linear minimization to be found. In Section~\ref{sec:metric_equivalence}, we introduce a readily computable equivalent metric defined in terms of the functions parametrizing the boundary of $\Omega$.


	\end{enumerate}
\end{rem}

 As an application of Theorem~\ref{thm-10-1-00}, we  deduce  Marcinkiewicz type  inequalities  on $C^2$-domains.  These inequalities  provide a basic tool for the discretization of the $L^p$-norm and are widely used in the study of the convergence properties of Fourier series, interpolation processes and orthogonal expansions. Interesting results  on Marcinkiewicz  type inequalities  for univariate polynomials can be found in \cite{Lu1, Lu2, Lu3, MT2, Er}, while results    for multivariate polynomials  on some special   multidimensional domains,  such as the  Euclidean balls, spheres,  polytopes, cones, spherical sectors, toruses,  were established in \cite{Kr2, MK, Da06}. The Bernstein type inequality stated in Theorem~\ref{thm-10-1-00} allows us  to extend these results  to  general  $C^2$-domains. For the special case $p=\infty$, such results have been shown to be valid under less restrictive smoothness requirements, see, e.g.~\cite{Kr19,O11}. More recently,    Marcinkiewicz type  inequalities were studied in a more general setting for elements of finite dimensional spaces (see \cite{DPTT,DPTTS-CA,DPTTS-JFA,Ko}).

  We need some notations and a definition.  
  Let $\rho: \Og\times \Og\to [0,\infty)$ be the metric on $\Og$ defined by 
  \begin{equation}\label{metric}\rho (\xi,\eta)\equiv \rho_\Og(\xi,\eta):=\|\xi-\eta\|+ \Bl|\sqrt{\dist(\xi, \Ga)} -\sqrt{\dist(\eta, \Ga)}\Br|,\   \  \xi, \eta\in\Og. \end{equation}
  For $\xi\in\Og$ and $t>0$, we define
  $$ U( \xi, t):= \{\eta\in\Og:\  \  \rho(\xi,\eta) \leq t\},  \  \   \  U^\circ ( \xi, t):= \{\eta\in\Og:\  \  \rho(\xi,\eta) < t\}.$$
  By Corollary~ \ref{rem-6-2} in Section ~\ref{sec:metric_equivalence}, we have   
  \begin{equation}\label{4-5-00}
  	|U(\xi, t)| \sim  t^d(\sqrt{\dist(\xi, \Ga)} +t),\  \  \xi\in\Og,\   \    \   \   0<t\leq 2,
  \end{equation}
  with the constant of equivalence depending only on $\Og$. In particular, this  implies
  \begin{equation*}\label{doubling}
  	|U(\xi, 2t)|\leq C |U(\xi,t)|,\   \ \forall \xi\in \Og,\  \ \forall t>0,
  \end{equation*}
  where the constant $C>0$ depends only on $\Og$. 
  Moreover,  according to  Corollary ~\ref{rem-6-2}~(ii),  every ball $U(\xi, L t)$ with $L>1$ and $t>0$ can be covered with $m\leq C L^d$ balls $U(\xi_j, t)$, $j=1,\dots, m$. 
    
 \begin{defn} A finite collection $\mathcal{R}=\{R_1, \dots, R_N\}$ of pairwise disjoint subsets of $\Og$ is called a  partition of $\Og$ with  norm   $\leq \da$   if $\Og=\bigcup_{j=1}^N  R_j$ and for each $1\leq j\leq N$,  $ R_j\subset U(\xi_j, \da)$ for some $\xi_j\in R_j$.  A partition  $\{R_1, \dots, R_N\}$ of $\Og$  is said to be a regular partition with norm  $\da$ if it satisfies 
 	 that 
 	\begin{equation}\label{regular}
 	U^\circ (\xi_j,  \da / 2) \subset R_j \subset U(\xi_j, \da),\   \ \forall 1\leq j\leq N.
 	\end{equation} 
 	
 \end{defn}

We give two remarks on this definition. 

 \begin{rem}

 A regular   partition of $\Og$ can be constructed through the so called maximal separated subsets of $\Og$. Indeed, for  $\da>0$,   a  set $\Ld\subset \Og$ is called $\da$-separated (with respect to the metric $\rho$)   if $\rho (\og, \og')\ge \da$ for any two distinct $\og, \og'\in\Ld$, while a $\da$-separated subset $\Ld\subset \Og$ is called maximal if 
 $\Og=\bigcup_{\og\in\Ld} U(\og, \da).$
Given  a maximal $\da$-separated subset   $\{\xi_j\}_{j=1}^N\subset \Og$, we define 
  \begin{align*}
  &R_1= \bigcap_{j=2}^N \Bl(U(\xi_1,\da)\setminus U^\circ  (\xi_j, \da/ 2)\Br),\\
  &  R_j = \Bl( \bigcap_{i:\  
  		1\leq i\neq j\leq N} U(\xi_j, \da) \setminus   U^\circ (\xi_i,\da/2)\Br) \setminus     \bigcup_{i=1}^{j-1} R_i  ,\   \ j=2,3,\dots, N.
  \end{align*}
  It is easily seen that  $\{R_1, \dots, R_N\}$ is  a regular  partition of $\Og$ with norm $\da$. 
 
\end{rem}
 
 \begin{rem}
 	If   $\{R_1, \dots, R_N\}$ of $\Og$  is a regular partition   of $\Og$ with norm  $\da\in (0,1)$, then one  must satisfy  $N\sim \da^{-d}$, with the constants of equivalence depending only on $\Og$. 
 	To see this, let $\xi_j\in R_j$, $1\leq j\leq N$ be such that \eqref{regular} is satisfied. Then the set $\{\xi_j\}_{j=1}^N$ is $\f\da 2$-separated, and hence by Corollary ~\ref{rem-6-2} (iii), we have  $N\leq C_\Og \da^{-d}$. 
 	On the other hand, by \eqref{4-5-00} and \eqref{regular}, we have 
 	\[ |\Og|=\sum_{j=1}^N |R_j| \leq \sum_{j=1}^N |U(\xi_j, \da)|\leq C \da^d \Bl( 1 +\sqrt{\diam (\Og)}\Br) N,\]
 	which implies  the lower estimate $N\ge  C_\Og \da^{-d}$.

 \end{rem}

 Now we can state the  Marcinkiewicz type inequality on a $C^2$-domain  as follows:
\begin{thm}\label{cor-16-2-0} Given    a compact $C^2$-domain $\Og\subset\RR^d$, there  exists a constant  $\da_0\in (0,1)$ depending  only on $\Og$   such that if   $\{R_1, \dots, R_N\}$ is a  partition of $\Og$ with  norm   $\leq \f {\da_0} n$ for some $n\in\NN$, then for any points  $\xi_j\in R_j$,  $1\leq j\leq N$, and any $1\le p\le \infty$, we have  
	\begin{align}
	\f 12 \|f\|_{L^p(\Og)} \leq \bigl( \sum_{j=1}^N |f(\xi_j)|^p  |R_j| \bigr)^{\f1p}\leq \f 32 \|f\|_{L^p(\Og)},\    \ \forall f\in \Pi_n^d, \label{MZ}
	\end{align}
where we  need to replace the $\ell^p$-norm with 	$\max_{1\leq j\leq N} |f(\xi_j)|$ if $p=\infty$.	
\end{thm}
%
%
%
%
%

\begin{rem}
While the Bernstein inequality  in Theorem \ref{thm-10-1-00} holds for the full range of $0<p\leq \infty$,   our proof of the Marcinkiewicz  inequality \eqref{MZ}, which   relies on H\"older's  inequality,  fails for $0<p<1$. 
\end{rem}

We will also apply   Theorem ~\ref{thm-10-1-00}  to show the existence of ``good'' positive cubature formulas on $C^2$-domains:

\begin{thm}\label{cor-16-3-0}  Given    a compact $C^2$-domain $\Og\subset\RR^d$, there  exists a constant  $\da_0\in (0,1)$ depending  only on $\Og$  such that if   $\{R_1, \dots, R_N\}$ is a  partition of $\Og$ with  norm  $\leq \f {\da_0} n$ for some positive integer $n$,  then for any points  $\xi_j\in R_j$,  $1\leq j\leq N$,
	there exist weights $\ld_j>0$, $1\leq j\leq N$ such that 
	$$\f 14 |R_j|\leq \ld_j \leq C \left|U\left(\xi_j, \f 1n\right)\right|,\   \  j=1,2,\dots,N$$
	and
	\begin{equation}\label{Cubature}
	\int_{\Og} f(\eta)\, d\eta =\sum_{j=1}^N \ld_j f(\xi_j),\   \   \  \forall f\in \Pi_n^{d},
	\end{equation} 	
	where $C>0$ is a constant depending only on $\Og$.
\end{thm}

Note that both Theorem \ref{cor-16-2-0} and Theorem \ref{cor-16-3-0} apply to general partitions of the domain $\Og$. If, in addition, $\{R_1,\dots, R_N\}$ is  a regular partition of $\Og$  with norm $\f \da n$, then  the number of  required points  $\xi_j$  in both \eqref{MZ} and \eqref{Cubature}  has the optimal  asymptotic order $n^d\sim \dim \Pi_n^d$ as $n\to\infty$, and the weights in  \eqref{MZ} and \eqref{Cubature} satisfy 
	$$\ld_j \sim |R_j|\sim  \Bl(\f {\da} n\Br)^d\Bl( \f \da n +\sqrt{\dist(\xi_j,\p \Og)} \Br) ,\   \  j=1,2,\dots,N.$$

Now we present the structure of the paper and the methods used. In Section~\ref{decom-lem}, we give the precise definition of $C^2$-domains,  and introduce a certain class of  $C^2$-domains,  called  domains  of special type, which  have simpler  boundary structure. More importantly, we  prove a decomposition proposition,   Proposition~\ref{lem-2-1-18}, which asserts that  every compact  $C^2$-domain   can be decomposed as a finite union of  domains of special type.

After that, in Section \ref{sec:metric_equivalence}, we introduce a new metric $\wh{\rho}_G$  on a domain $G$ of special type, and prove that $\wh{\rho}_G$  is equivalent to  the restriction of the metric   $\rho_{\Og}$ on $G$ if $G\subset \Og$ is attached to  the boundary  $\p\Og$.  This   new metric  $\wh \rho_G$  has the advantage that it is easier to deal with. Moreover,  the equivalence of  $\wh{\rho}_G$ with $\rho_\Og$  combined  with the decomposition proposition for $C^2$-domains 
allows us to effectively reduce  considerations near the boundary to certain problems on domains of special type.

After these preparations, we prove our main result, Theorem~\ref{thm-10-1-00}, in Sections~\ref{sec:12}-\ref{sec:14}.  The proof is long and rather involved, so we break it into several parts.  A  crucial part  is given   in Section ~\ref{sec:12}, where we establish    a Bernstein type inequality on domains of special type in $\RR^2$.  
An important ingredient in our proof   is to   construct a family   of parabolas   touching the boundary $\Ga=\p \Og$ and lying  inside the domain $\Og$, for which   every  point $(x,y)\in \Og$ near the boundary $\Ga$ can  be connected with a unique boundary point  through   one of these parabolas. 
 In many cases, we may use   these parabolas to replace  the usual line segments that are parallel to one of the coordinate axis. Indeed,  performing a change of variables, we prove in Section ~\ref{sec:12} that every  double   integral near the boundary of $\Og$ can be expressed   as   iterated integrals along the family of parabolas. This technique plays a crucial role in the proof of the  Bernstein inequality  along tangential directions on the boundary $\p \Og$. 
 
 In Section  ~\ref{sec:13},
we   show how  the Bernstein type inequality 
 on  higher-dimensional  domains of special type can be deduced from the corresponding result for $d=2$.    One of  the main difficulties in the higher-dimensional case   comes from the fact that we have  to deal with certain non-commutative   mixed directional derivatives along different  tangential directions on the boundary.

 In   Section ~\ref{sec:14},  we prove  Theorem ~\ref{thm-10-1-00}, using the decomposition  proposition,  and the  results on domains of special type that have already been proven in  Sections~\ref{sec:12}-\ref{sec:13}.
 
  Let us now provide some brief comments  regarding earlier relevant works. A different approach to the problem of estimating derivatives of algebraic polynomials is based on pluripotential theory, see, e.g. the fundamental work~\cite{Ba} and the paper~\cite{O4} which also discusses the connection with the ``real geometric'' approach pursued here. The natural idea of reducing the multivariate problem to the univariate one by considering restrictions to certain smaller-dimensional subsets appears already in~\cite{O14} where restrictions to segments yield Markov-type estimates for gradients of polynomials. This would not suffice for Bernstein-type estimates and considering higher-order restrictions, such as ellipses or parabolas, is a natural choice. For example, in~\cite{O13}, ellipsoids were used to obtain strong Bernstein-Markov type inequalities, while parametrization by curves of degree two was employed in~\cite{O3,O4}. In contrast to the above works, our primary focus in the current work is \emph{integral} norms, where the associated family of degree two curves (parabolas in our case) needs to cover the domain in a uniform manner allowing an appropriate change of variable.

 The proofs of  Theorem ~\ref{cor-16-2-0} and Theorem ~\ref{cor-16-3-0} are given in the last section, Section ~\ref{sec:16}.  The Bernstein inequality stated in Theorem ~\ref{thm-10-1-00} plays a crucial role in these proofs.


Finally, we  point out  that the  Bernstein inequality stated in  Theorem~\ref{thm-10-1-00} has many  applications besides the applications presented in Section~\ref{sec:16}. 
These  include the inverse theorem on the error of polynomial approximation in terms of moduli of smoothness, and the Chebyshev type cubature formulas on $C^2$-domains.
We will return to these topics  in an upcoming paper.

\section{Decomposition of $C^2$-domains into domains of special type}\label{decom-lem}

Our aim in this section is to show that every compact  $C^2$-domain   can be decomposed as a finite union of  domains of special type, whose definition   is given as follows: 
\begin{defn}
	A set    $G\subset \RR^{d}$ is  called an {\sl upward}  $x_{d}$-domain with base size $b>0$ and parameter $L\ge 1$  if there exist $\xi\in\RR^d$ and $g\in C^2(\R^{d-1})$ such that
\begin{equation*} \label{2-7-special}G=\xi+\Bl\{( x,  y)\in\RR^d:\  \  x\in (-b,b)^{d-1},\   \  g(x)- L b< y \leq g(x)\Br\}.\end{equation*}
In this case, we define 
\begin{align}
G^\ast:&=\xi+\Bl\{( x,   y):  x\in (-2b,2b)^{d-1},   \min_{u\in [-2b, 2b]^{d-1}} g(u)-4Lb <  y \leq g(x)\Br\}.\label{G}
\end{align}
Furthermore, given  a  parameter  $\ld>0$,  we define 
\begin{align*}
G(\ld):&=\xi +\Bl\{ (x,   y):\  \  x\in (-\ld b, \ld b)^{d-1},\   \   g(x)-\ld L b < y \leq g(x)\Br\},\end{align*}
and call the set \begin{align*}
\p'G(\ld)&:=\xi +\Bl\{ (x,  g(x)):\  \  x\in (-\ld b, \ld b)^{d-1}\Br\}
\end{align*} 
the essential boundary of the set $G(\ld)$.

\end{defn}

Several remarks on this definition are in order.

\begin{rem}\label{rem-2-1-0}
\begin{enumerate}[\rm (a)]
	\item With a possible change of the point $\xi\in\R^d$, we may always assume that the function $g$ satisfies 
	$$\min_{x\in [-2b,2b]^{d-1}} g(x) =4Lb.$$	
	\item For technical reasons, sometimes  we  may need to  choose the base size  $b$    as small as we wish, and   the parameter $L$ large enough so that  
	\begin{equation}\label{parameter-2-9}
 	L\ge 4\sqrt{d} \max_{x\in [-2b, 2b]^{d-1}} \|\nabla g(x)\| +1.
	\end{equation}

\end{enumerate}	 

\end{rem}

In general, given $1\leq j\leq d$, we  may    define 
an upward  or downward $x_j$-domain $G\subset \RR^{d}$    and the corresponding    sets $G(\ld)$,   $\p' G(\ld)$, $G^\ast$   in a similar way, using  the reflections:  for $x=(x_1, x_2,\dots, x_d)\in\R^d$,  
\begin{align*}
\sa_j (x):& =(x_1,\dots, x_{j-1}, x_{d}, x_{j+1},\dots, x_{d-1}, x_j),\\
\tau_j (x) :&=(x_1, \dots, x_{j-1}, -x_j, x_{j+1}, \dots, x_{d}).
\end{align*}

\begin{defn}\label{Def-2-1} \begin{enumerate}[\rm (i)]
		\item 	A set  $G\subset \RR^{d}$
	is called  an  {\it upward} $x_j$-domain with base size $b>0$  and parameter  $L\ge 1$ if its reflection  $E:=\sa_j (G)$ is an  upward $x_{d}$-domain with base size $b$ and parameter $L$, in which case we define   
	$$G (\ld) = \sa_j \bl( E (\ld)\br),\    \ 
	\p' G(\ld)= \sa_j\bl(\p' E (\ld)\br),\   \   \   G^\ast =\sa_j (E^\ast).$$ 
	\item  A set $G\subset \RR^{d}$
	is called  a  {\it downward} $x_j$-domain with base size $b>0$ and parameter $L\ge 1$  if its reflection  $H:=\tau_j (G)$ is an  upward $x_{j}$-domain with base size $b$ and parameter $L\ge 1$, in which case we define 
	$$G (\ld) = \tau_j \bl( H (\ld)\br),\   \
	\p' G(\ld)= \tau_j\bl(\p' H (\ld)\br),\   \    G^\ast =\tau_j(H^\ast).$$

\end{enumerate}
\end{defn}

Now we  are in a position to define a domain of special type  attached to the boundary of  a $C^2$-domain. Let $e_1=(1,0,\dots, 0), \dots, e_d =(0, \dots, 0, 1)\in\R^d$ denote the standard canonical basis  in $\RR^d$. A   rectangular box in $\RR^{d}$ is a set  of  the form $[a_1,b_1]\times \dots\times [a_{d}, b_{d}]$ with $-\infty<a_j<b_j<\infty$,  $j=1,\dots, d$.  
Here and throughout the paper, we  always assume that the sides of a  rectangular box   are parallel with  the coordinate axes.   

\begin{defn}\label{def:specialtype}  \begin{enumerate}[\rm (i)]
		\item  A set     $G\subset  \RR^{d}$ is called    a domain  of special type   if there exists $1\leq j\leq d$ such that  it  is either  an upward or a  downward $x_j$-domain. In this case,    we call $\p' G(\ld)$ the essential boundary of $G(\ld)$, and set      $\p' G: =\p' G(1)$ and $\p' G^\ast :=\p' G (2)$.
		\item Let $\Og$ be a compact  $C^2$-domain in $\R^d$ and   $G\subset \Og$ a domain of special type. We say 
	 $G$  is   attached to the boundary $\p\Og$ of $\Og$   if  
		 $\overline{G^\ast}\cap \p\Og =\overline{\p' G^\ast}$ 
		and there exists an open rectangular box $Q$ in $\RR^{d}$  such that $ G^\ast =Q\cap \Og$.
	\end{enumerate}

\end{defn}

With the above definitions, we can  now state the  main result  in this section  as follows:

\begin{prop}\label{lem-2-1-18} Given    a compact  $C^2$-domain $\Og\subset\RR^d$,  there exists a finite cover of the boundary  $\Gamma:=\p \Og$ of $\Og$  by domains of special type  $G_1, \dots, G_{m_0}\subset \Og$  attached to $\Ga$.  Moreover, we may select these  domains    in  a way such  that    the base   size of each  $G_j$  is as small as we wish, and  the  parameter  of  each  $G_j$  satisfies  the condition~\eqref{parameter-2-9}.
\end{prop}

\begin{rem}\label{rem-2-7}
	From the  proof of Proposition \ref{lem-2-1-18} below, it is easily seen that for each parameter $\ld_0\in (0,1]$, we can find  $m_0\leq C(\Og, \ld_0) $ domains  $G_1,\dots, G_{m_0}\subset \Og $ of special type attached to the boundary  $\Ga$ of a given  $C^2$-domain $\Og$  such that 
	$\Ga\subset \bigcup_{j=1}^{m_0} G_j(\ld_0)$.  The point here is  that we can choose  the parameter $\ld_0\in (0,1]$ as small as we wish. 
\end{rem}
	
	\begin{rem}
		By Definition \ref{def:specialtype}, for each domain $G_j$ given in Proposition \ref{lem-2-1-18}, we have $G_j(\ld) \subset G^\ast_j \subset \Og$ for all $0<\ld\leq 2$. 
	\end{rem}

\begin{proof} 
	
	Since the essential boundary  $\p' G$ of a domain $G\subset \Og$ of special type attached to $\Ga$ is  open relative to the topology of $\Ga$,  it  suffices to show that for  each fixed    $\xi\in \Ga$,  there exists     a domain $G_\xi$ of special type attached to $\Ga$    such that $\xi\in \p' G_\xi $,  and such that  its  base size  is as small as we wish, and  its   parameter    satisfies  \eqref{parameter-2-9}.	
	Without loss of generality, we may assume that $\xi=0\in\Ga$  since otherwise we may translate  the domain $\Og$ so that $\xi$ coincides with the origin. Let $\mathbf n_0$ denote the unit outer normal vector to $\Ga$ at the origin. 
	 Without loss of generality, we may also assume that    $\mathbf  n_0\cdot~ e_{d}=\max_{1\leq i\leq d} |\mathbf  n_0\cdot e_i|$, since   the other cases can be treated similarly. Indeed,  if $1\leq j\leq d$ is such that $|\mathbf n_0\cdot e_j| =\max_{1\leq i\leq d} |\mathbf n_0\cdot e_i|$, then a slight modification of   the construction of the set   $G_0$   below would  lead to  an upward or a downward $x_j$-domain attached to $\Ga$ according to whether $\mathbf n_0\cdot e_j>0$ or $\mathbf n_0\cdot e_j<0$.

	 	Since $\Og$ is a $C^2$-domain,  
	 there exists a number  $r_0\in (0,1)$  depending only on $\Og$ such that    for each $\eta\in\Ga$, 
	 \begin{equation}\label{2-6b} B_{8dr_0} [\eta-8dr_0 \mathbf n_\eta]\subset \Og\   \ \text{and}\     \   B_{8dr_0} ( \eta+8dr_0 \mathbf  n_\eta) \subset \Og^c=\RR^d\setminus \Og,\end{equation}
	 where $\mathbf n_\eta$ denotes the unit outer normal vector to $\Ga$ at $\eta$.

	Next, since  $\mathbf n_0\in\sph$,   we have  $\mathbf  n_0\cdot e_{d}=\max_{1\leq i\leq d} |\mathbf  n_0\cdot e_i| \ge \f 1{\sqrt{d}}$. Thus,   by  the implicit function theorem, there exist  an open  rectangular box  $V_0:=I_0  \times (-a_0, a_0)$ centered at $0\in\Ga$ with   $I_0:=(-\delta_0, \delta_0)^{d-1}$  and $ a_0>0$,  and a $C^2$-function $h$ on $\RR^{d-1}$ such that $h(0)=0$,
$h(I_0) \subset (-a_0, a_0)$,  and 
	the surface 	$\Ga_0:=\Ga\cap V_0$  can be represented as 
		$$ \Ga_0=\Bl\{(x, h(x)):\  \  x\in I_0\Br\}.$$
	  By continuity, we may choose the constant $\da_0$ small enough so that $0<\da_0<r_0$,    
		$\mathbf n_\eta\cdot e_{d} \ge \f 1{2\sqrt{d}}$ for every $\eta\in \Ga_0$,
		and 
			\begin{equation}\label{5-1-0-18}
		\|\nabla h\|_{L^\infty(I_0)}<  r_0  \delta_0^{-1}.
		\end{equation}
	In particular,  \eqref{5-1-0-18} implies that 
	\begin{equation}\label{2-8b}
	|h(x)|< \sqrt{d} r_0,\   \  \forall x\in I_0.
	\end{equation}	
	Using \eqref{2-6b} with $\eta=(x, h(x))\in\Ga_0$ and $x\in I_0$, and taking into account the fact that $\mathbf n_\eta\cdot e_d \ge \f 1 {2\sqrt{d}}$, we obtain via  a simple geometric argument  that 
	\begin{align*}
	 & \Bl\{ (x,y):\    x\in I_0,\   h(x)<y<h(x) +8\sqrt{d} r_0\Br\}\subset 	 
	  \Og^c, 
	\end{align*}	
	and 
		\begin{align*}
	& \Bl\{ (x, y):\  \  x\in  I_0,\   \   h(x)-8\sqrt{d} r_0 \leq y  <  h(x)\Br\}	\subset \Og^\circ, \end{align*} 
	where $\Og^\circ$ denotes the interior of $\Og$. 	
	It then follows by  \eqref{2-8b} that 
\begin{align}
&\Bl\{ (x,y):\    x\in I_0,\   h(x)<y<7\sqrt{d} r_0\Br\}\subset \Og^c,\label{2-9b}\\
\text{and}\  \  \ &	\Bl\{ (x, y):\  \  x\in  I_0,\   \   -7\sqrt{d} r_0 \leq y  <  h(x)\Br\}\subset \Og^\circ. \label{claim-2-9}
\end{align}		
		Now setting   $$\varsigma_0:=(0, -7\sqrt{d}r_0),\   \ 
	g_0(x):= h ( x)+7\sqrt{d}r_0,$$
	we obtain from \eqref{2-9b} and \eqref{claim-2-9} that \begin{align}
&	\varsigma_0+\Bl\{ (x,y):\  x\in I_0,\  \ g(x) <y <14\sqrt{d} r_0\Br\}\subset \Og^c,\label{2-11b}\\
	&	\varsigma_0+\Bl\{ (x,y):\  x\in I_0,\ 0\leq y<g(x)\Br\} \subset \Og^\circ.\label{2-12b}
	\end{align}
		Furthermore, 	by~\eqref{5-1-0-18}, we have 
	\begin{equation}\label{2-13b}
	6\sqrt{d} r_0<g_0(x)<8\sqrt{d}r_0,\    \ \forall x\in I_0.
	\end{equation}
	
	Third, we  define 
  $$G_0: =\varsigma_0+  \Bl\{ ( x, y):  \  \ x\in (- b_0, b_0)^{d-1},\   \  g_0(x)-\da_0<  y \leq g_0(x)\Br\},$$
	where   	  $b_0\in (0, \da_0/2)$  is a  constant  such  that $$\f {\da_0} {b_0} \ge 4\sqrt{d}  \|\nabla h\|_{L^\infty(I_0)} +1.$$
Clearly, $G_0$ is 	an upward $x_{d}$-domain with base size $b_0$ and parameter   $L_0:=\da_0 / b_0$.
 Since $[-2b_0, 2b_0]^{d-1}\subset I_0$ and $\da_0<r_0$,  we obtain from \eqref{2-13b} that   
$$ \al_0:= \min _{x\in [-2b_0, 2b_0]^{d-1}} g_0(x) -4L_0 b_0\ge 6\sqrt{d} r_0 -4\da_0>2r_0>0.$$ 
It then  follows from \eqref{2-12b}  that  
\begin{align*}G_0^\ast &=\varsigma_0+  \Bl\{ ( x, y):  \  \ x\in (- 2b_0,2 b_0)^{d-1},\   \  \al_0<  y \leq g_0(x)\Br\}\\
 &\subset \varsigma_0 +\Bl \{ (x, y):\  \  x\in  [-\delta_0, \delta_0]^{d-1},\   \  0 \leq y  \leq  g_0(x)\Br\} \subset \Og.\end{align*}
 
Finally,  \eqref{2-12b} implies that 
$$ \overline{ G_0^\ast} \cap \Ga =\overline{\p' G_0^\ast}$$
while \eqref{2-12b} together with \eqref{2-11b}  implies that
$$ G_0^\ast =\Og\cap Q_0$$ with 
$Q_0=\varsigma_0 + (-2b_0, 2b_0)^{d-1} \times (\al_0, 14\sqrt{d} r_0)$.   Thus,   $G_0\subset \Og$ is an upward  $x_{d}$-domain attached to $\Ga$. 	
	 This completes the proof of  Proposition ~\ref{lem-2-1-18}.
\end{proof}

\section{Metrics on domains of special type}\label{sec:metric_equivalence}

Let $\rho_\Og:\Og\times \Og\to [0,\infty)$  be the  metric on $\Og$ given in \eqref{metric}.
In this section, we shall  introduce a new  metric $\wh \rho_G$ on  a domain $G$  of special type, which is equivalent to  the restriction of  $\rho_{\Og}$ on $G$ if $G\subset \Og$ is attached to  $\Ga:=\p\Og$.  This  new metric  $\wh \rho_G$ is easier to deal with in  applications.

To be precise, let $G\subset \R^d$ be an $x_d$-upward domain with base size $b\in (0,1)$ and parameter $L>0$:  
 \begin{align*}
 G:=\varsigma+\{ (x, y):\  \  x\in (-b,b)^{d-1},\   \  g(x)-Lb<  y\leq  g(x)\},\   \ \varsigma\in\RR^d,
 \end{align*}
 where    $g$ is a $C^2$-function on $\RR^{d-1}$. Then 	$$G^\ast=\varsigma+\Bl\{ (x, y):\  \  x\in (-2b,2b)^{d-1},\   \  \min_{u\in [-2b, 2b]^{d-1}} g(u)-4Lb < y\leq  g(x)\Br\}$$   
 and we define a metric $\wh\rho_G: \overline{G^\ast}\times \overline{G^\ast} \to (0,\infty)$ by 
\begin{equation}\label{rhog}
\wh{\rho}_G(\varsigma+\xi, \varsigma+\eta):=\max\Bl\{\|\xi_x-\eta_x\|,
\Bl|\sqrt{g(\xi_x)-\xi_y}-\sqrt{g(\eta_x)-\eta_y}\Br|\Br\}
\end{equation}
for all $ \xi=(\xi_x, \xi_y),\eta=(\eta_x,\eta_y)\in \overline{G^\ast}-\varsigma$.  Here and throughout this paper,  we often     use Greek letters $\xi,\eta,\al,\dots$ to denote points in $\RR^{d}$ and write $\xi\in\RR^d$ as  $\xi=(\xi_x, \xi_y)$ with $\xi_x\in\RR^{d-1}$ and $\xi_y\in\RR$.
Finally, we can define the metric $\wh{\rho}_G$   on a more general  $x_j$-domain  $G\subset \RR^d$ (upward or downward) in a similar way.

Our aim in this section is to show that  the metric $\wh\rho_G$ defined above  is equivalent to the restriction of $\rho_\Og$ on $G$ when $G\subset\Og$ is attached to $\Ga=\p \Og$. 

\begin{prop}\label{metric-lem} If   $G \subset \Og$  is   a domain of special type attached to $\Ga$, 
	 then 
	\begin{equation*}\label{6-1-metric}\wh{\rho}_G(\xi,\eta)\sim \rho_{\Og} (\xi,\eta),\    \    \  \xi, \eta\in G\end{equation*}
	with the constants of equivalence depending only on $G$ and $\Og$.
\end{prop}
\begin{proof}
	Without loss of generality, we may assume that 
	\begin{align*}\label{standard}
		G:=\{ (x, y):\  \  x\in (-b,b)^{d-1},\   \  g(x)-Lb<  y\leq  g(x)\},\  \ b\in (0,1),\  \ L\ge 1,
	\end{align*}
	and  $g\in C^2(\RR^{d-1})$ satisfies  $\min_{x\in [-2b, 2b]^{d-1}} g(x)=4Lb$. Indeed, a slight modification of the proof below works equally well  for more general domains of special type.
	
	Since $G$ is attached to $\Ga$, it follows that 
	$$  \Ga':=\{ (x, g(x)):\   \  x\in [-2b, 2b]^{d-1}\}
	\subset \Ga.$$

	The following lemma plays an important role in the proof of Proposition ~\ref{metric-lem}:
	\begin{lem}\label{lem-9-1} If $\xi=(\xi_x, \xi_y)\in  G$, then 
		$$c_\ast (g(\xi_x)-\xi_y)\leq \dist(\xi, \Ga')\leq g(\xi_x)-\xi_y,$$
		where 
		$ c_\ast =  \f 1 {3 \sqrt{ 1+\|\nabla g\|_\infty^2}}$
		and   $\|\nabla g\|_\infty=\max_{x\in [-2b,2b]^{d-1}}\|\nabla g(x)\|$. 
	\end{lem}
	
	\begin{proof}
		Let  $\xi=(\xi_x, \xi_y)\in G$. Since $(\xi_x, g(\xi_x))\in \p' G\subset \Ga'$, we have 
		$$\dist(\xi, \Ga')\leq \|(\xi_x,\xi_y)-(\xi_x, g(\xi_x))\|=g(\xi_x)-\xi_y.$$
		
		It remains to prove  the inverse inequality,
		\begin{equation}\label{3-4a}
	 \dist(\xi, \Ga')\ge 	c_\ast (g(\xi_x)-\xi_y).
		\end{equation}
	Let   $(x, g(x))\in\Ga'$ be such that 
	$$\dist(\xi, \Ga') =\|\xi-(x,g(x))\|.$$
	Since   
	$$ \dist(\xi, \Ga')\ge \| x-\xi_x\|,$$
	\eqref{3-4a} holds trivially if  $\|x-\xi_x\|\ge   c_\ast (g(\xi_x)-\xi_y)$. Thus, without loss of generality, we may assume that 
	$\|x-\xi_x\|< c_\ast (g(\xi_x)-\xi_y)$. We then write 	
		\begin{align}
			\|\xi-(x,g(x))\|^2=&\|\xi_x- x\|^2 +|\xi_y -g(\xi_x)|^2+\notag\\
		&+	|g(\xi_x)-g(x)|^2
			  +2 (\xi_y-g(\xi_x))\cdot (g(\xi_x)-g(x)).\label{eq-9-2}
		\end{align}
	Since  $\|x-\xi_x\|\leq  c_\ast (g(\xi_x)-\xi_y)$, we have 
		\begin{align*}
			& \|\xi_x- x\|^2+
			|g(\xi_x)-g( x)|^2+2 (\xi_y-g(\xi_x))\cdot (g(\xi_x)-g( x))\\
			&\leq (1+\|\nabla g\|_\infty^2) \|\xi_x-x\|^2 +2\|\nabla g\|_\infty (g(\xi_x)-\xi_y)\| \xi_x- x\|\\
			&\leq \Bl[c_\ast^2 (1+\|\nabla g\|_\infty^2)+ 2 \|\nabla g\|_\infty c_\ast\Br] (g(\xi_x)-\xi_y)^2\leq \f 79 (g(\xi_x)-\xi_y)^2.
		\end{align*}
	Thus, using~\eqref{eq-9-2},   we obtain
		$$
	\dist(\xi, \Ga')^2=	\|\xi-(x,g(x))\|^2
		\ge \f29 |\xi_y -g(\xi_x)|^2,$$
		which implies the desired lower estimate \eqref{3-4a}. This completes the proof of Lemma~\ref{lem-9-1}.
	\end{proof}

	Let us return to  the proof of Proposition~\ref{metric-lem}.  Set
	$$\Ga'':=\Bl\{(x,g(x)):\   \ x\in \Bl[-\f {3b}2, \f {3b}2 \Br]^{d-1}\Br\}\subset \Ga'.$$
		Since $G\subset \Og$ is attached to $\Ga$, 
	there exists an  open rectangular box $Q\subset \RR^{d}$ such that 
	$$G^\ast=\{ (x, y):\  \  x\in (-2b,2b)^{d-1},\   \  0< y\leq  g(x)\}= \Og\cap Q.$$
	In particular, this   implies that there exists a small constant $\va_0\in (0, b/4)$ depending only on $\Og$ and $G$ such that  $\dist(\xi, \Ga\setminus \Ga'') \ge \va_0$ whenever 
	$\xi\in G$.

	Let $\va\in (0,\va_0)$ be a  small constant to be specified later. 	
	Let   $\xi=(\xi_x, \xi_y), \eta=(\eta_x, \eta_y)\in G$.
	From  Lemma~\ref{lem-9-1},  it is easily seen that if  $\max\{\dist(\xi, \Ga), \dist(\eta, \Ga)\}\ge \va$  or $\|\xi-\eta\|\ge \va$, then 
	$$ \rho_{\Og} (\xi, \eta) \sim \wh{\rho}_G (\xi, \eta)\sim \|\xi-\eta\|,$$
	where the constants of equivalence depend on $\va$. 
	Thus, it suffices to prove  that there exists a constant $\va\in (0, \va_0)$ depending only on $\Og$ and $G$ such that  
	\begin{equation}\label{3-6a}
	\rho_{\Og} (\xi, \eta) \sim \wh{\rho}_G (\xi, \eta)
	\end{equation}
if 
	\begin{equation}\label{3-6b}
	\dist(\xi, \Ga)<\va,\  \  \dist(\eta,\Ga) < \va\   \   \text{and}\   \ \|\xi-\eta\|< \va.
	\end{equation}


	To show this,  we need to  introduce a few  notations. Given $\al=(\al_x,\al_y)\in G$, we set 
	$s_\al:=\dist(\al,\Ga'')$,  and denote by  
	  $t_\al$ the point  in $ [-\f {3b}2, \f{3b}2]^{d-1}$   such that 
	\begin{equation}\label{3-6c}
	s_\al=\|\al-(t_\al, g(t_\al))\|.
	\end{equation}
	Since $\Ga$ is a $C^2$-surface,  a straightforward calculation shows that every  $\al=(\al_x, \al_y)\in G$ can be represented  as a function of $(t_\al, s_\al)$: 
	$$ \al = \Bl( x(t_\al, s_\al),\   \    g(t_\al) -s_\al A(t_\al)\Br),$$
	where 
	\begin{align}x(t,s):&=t+s A(t)\nabla  g(t), \    \    \    A(t)=(\sqrt{1+\|\nabla g(t)\|^2})^{-1}, \label{def-9-3}\\
	&  t\in [-2b,2b]^{d-1},\  \ s>0. \notag\end{align}
	 Thus,  for every $\al=(\al_x,\al_y) \in G$, 
	\begin{equation}\label{1-1}
		g(\al_x)-\al_y =F(t_\al, s_\al),
	\end{equation}
	where  $F$ is a function   on  the rectangular box   $ [-2b, 2b]^{d-1}\times [0, \va_0]$ given by 
	\begin{equation*}\label{def-9-3b}  F(t,s):=g\bl(x(t,s)\br)-g(t)+s A(t).\    \    \    \end{equation*}
	
	Now we turn to the proof of  \eqref{3-6a}. 
		Note first  that \eqref{3-6b} implies that 
	$$\dist(\xi,\Ga)=\dist(\xi,\Ga')=\dist(\xi,\Ga'')\   \   \text{and}\   \ \dist(\eta,\Ga)=\dist(\xi,\Ga')=\dist(\eta,\Ga'').$$
Thus,  
	$s_\xi, s_\eta\in [0,\va)$.  Furthermore, by \eqref{3-6c}, we have
	$$\|t_\xi-\xi_x\|\leq s_\xi<\va\   \ \text{and}\   \ \|t_\eta-\eta_x\|\leq s_\eta<\va. $$
	Since $\|\eta-\xi\|<\va$, it follows that 
	$\|t_\xi-t_\eta\|<2\va$.
	
	Now using Lemma~\ref{lem-9-1} and~\eqref{1-1}, we obtain  
	\begin{align}\label{3-12a}
	\Bl|\sqrt{g(\xi_x)-\xi_y}&-\sqrt{g(\eta_x)-\eta_y}\Br|
	\sim \f { \Bl|F(t_\xi, s_\xi)-F(t_\eta, s_\eta)\Br|}{\sqrt{s_\xi}+\sqrt{s_\eta}},\end{align}
	A straightforward calculation shows that if $\|t-t_\xi\|\leq 2\va$ and $s\in [0, \va]$, then 
	\begin{align}
		\f {\p F}{\p s}(t,s)&=\Bl[ 1+(\nabla g)(x(t,s))\cdot \nabla g(t)\Br]A(t)=\sqrt{ 1+\|\nabla g(t_\xi)\|^2} +O(\va),\label{9-3-0}\end{align}
	and \begin{align}
		\|\nabla_t F(t,s)\|
		&=\Bl\| \nabla g(x(t,s))-\nabla g(t)+s \nabla A(t)\notag\\
		&	+s\nabla g(x(t,s))\Bl[ (\nabla g(t))^{tr}\nabla A(t) + A(t) H_g(t)\Br]\Br\| \leq C s,\label{9-4-0}
	\end{align}
	where $\nabla_t $ denotes  the gradient operator $\nabla$  acting on the variable $t$, $\nabla g$ is treated   as a row vector, and 
	$H_g:=(\p_i\p_j g)_{1\leq i, j\leq d-1}$ denotes the Hessian matrix of $g$. 
	Thus, using  the mean value theorem,~\eqref{9-3-0} and~\eqref{9-4-0},  we obtain 
	\begin{align*} \f { \Bl|F(t_\xi, s_\xi)-F(t_\eta, s_\eta)\Br|}{\sqrt{s_\xi}+\sqrt{s_\eta}}=\Bl[ \sqrt{1+|\nabla g(t_\xi)|^2}+O(\va)\Br] |\sqrt{s_\xi}-\sqrt{s_\eta}| +O(\sqrt\va)  \|t_\xi-t_\eta\|,
	\end{align*}
	which  together with  \eqref{3-12a} implies that 
	\begin{equation}\label{3-15b}
		\Bl|\sqrt{g(\xi_x)-\xi_y} -\sqrt{g(\eta_x)-\eta_y} \Br|\sim |\sqrt{s_\xi} -\sqrt{s_\eta}|+O(\sqrt\va)  \|t_\xi-t_\eta\|\\
	\end{equation}
	provided that $\va\in (0, \va_0)$ is sufficiently small. 
	On the other hand,   using ~\eqref{def-9-3} and the mean value theorem, we have
	\begin{align*}
		 \|\xi_x-\eta_x\|&=\|x(t_\xi, s_\xi) -x(t_\eta, s_\eta)\|=
		\Bl\| t_\xi-t_\eta+s_\xi A(t_\xi) \nabla g(t_\xi)-s_\eta A(t_\eta) \nabla g(t_\eta)\Br\|\notag\\
		&= \|t_\xi-t_\eta\|+ O(\sqrt{\va}) |\sqrt{s_\xi}-\sqrt{s_\eta}|+O(\va) \|t_\xi-t_\eta\|,
	\end{align*}
	which implies 
	\begin{equation}
	\|\xi_x-\eta_x\|
	\sim  \|t_\xi-t_\eta\|+ O(\sqrt{\va}) |\sqrt{s_\xi}-\sqrt{s_\eta}|.\label{1-4}
	\end{equation}
	provided that $\va\in (0, \va_0)$ is small enough. 
	Thus, combining \eqref{3-15b} and \eqref{1-4}, we deduce 
	\begin{align*}
		\wh{\rho}_G(\xi,\eta) &\sim \Bl| \sqrt{g(\xi_x) -\xi_y} -\sqrt{ g(\eta_x) -\eta_y}\Br| +\|\xi_x-\eta_x\|\\
		&\sim |\sqrt{s_\xi}-\sqrt{s_\eta}|+O(\sqrt{\va}) \|t_\xi-t_\eta\| +\|\xi_x-\eta_x\|
			\sim |\sqrt{s_\xi}-\sqrt{s_\eta}|+\|\xi_x-\eta_x\|\\
		&\sim \rho_{\Og}(\xi,\eta)=\max\Bl\{ |\sqrt{s_\xi}-\sqrt{s_\eta}|,\  \ \|\xi-\eta\|\Br\},
	\end{align*}
where the last step uses the fact that 
	$$|\xi_y-\eta_y| \leq C\|\xi-\eta\|\leq  C \wh{\rho}_G (\xi,\eta).$$
	This completes the proof of Proposition~\ref{metric-lem}.
\end{proof}

 Recall that $$ U(\xi,\da):=\{\eta\in\Og:\  \  \rho_{\Og}(\xi,\eta)\leq \da\}.$$

We conclude this section with the following useful corollary:
\begin{cor}\label{rem-6-2}\begin{enumerate}[\rm (i)]
		\item 	For each $\xi\in\Og$ and $\da\in (0,1)$,
		\begin{equation}\label{3-17}|U(\xi, \da)|\sim \da^d \Bl( \da +\sqrt{\dist(\xi, \Ga)}\Br) \end{equation}
		with the constant of equivalence depending only on $\Og$.
		\item  There exists a constant $C>1$ depending only on $\Og$ such that for each $\xi\in \Og$, $\da>0$ and $L>1$, there exist $m\leq C L^d$ points  $\xi_1,\dots,\xi_m\in U(\xi, L\da)$  such that 
		\begin{equation*}\label{3-18a}
		U(\xi, L\da) \subset \bigcup_{j=1}^m U(\xi_j, \da).
		\end{equation*}
		\item   There exists a constant $C>1$ depending only on $\Og$ such that for each   $\da\in (0, 1)$, and each  $(\da,\rho_\Og)$-separated subset of $\Og$, we have $\# \Ld \leq C \da^{-d}$. 
		
	\end{enumerate}

\end{cor}
\begin{proof} (i)
	By Proposition ~\ref{lem-2-1-18} and Proposition~\ref{metric-lem},   it suffices  to prove  the property   for each domain $G\subset \Og$ of special type  attached to $\Ga$, and for  the metric $\wh{\rho}_G$ given in \eqref{rhog}.  Without loss of generality, we may assume that 
	 \begin{align*}
	G:=\Bl\{ (x, y):\  \  x\in (-b,b)^{d-1},\   \  g(x)-b<  y\leq  g(x)\Br\},
	\end{align*}
	where  $b>0$ and   $g$ is a $C^2$-function on $\RR^{d-1}$ satisfying $\max_{x\in [-2a, 2a]^{d-1}} g(x)=4b$.
By   Lemma~\ref{lem-9-1}, it is enough to show  that for each $\xi=(\xi_x, \xi_y) \in G$ and $0<\da <b/2$, 
	\begin{align}
&	\Bl|\Bl\{ (x,y)\in G^\ast:\  \ \|x-\xi_x\|\leq \da,\  \ | \sqrt{g(\xi_x)-\xi_y} -\sqrt{g(x)-y}|\leq \da\Br\}\Br|\notag\\
	& \sim \da^d ( \da +\sqrt{g(\xi_x)-\xi_y} ).\label{3-18}
	\end{align}
	Indeed, \eqref{3-18} can be easily verified by Fubini's theorem. 
	This proves  \eqref{3-17}.
	
	 Note that a similar argument also shows that for any $\da_1, \da_2>0$ and $\xi\in \Og$,
\begin{align}\label{3-20}
\Bl| \Bl\{ \eta\in\Og:\   \ \|\eta-\xi\| \leq \da_1,\   \  \dist(\eta, \Ga) \leq \da_2\Br\}\Br| \leq C \da_1^{d-1} \da_2.
\end{align}
	
	(ii)  Let $\Ld:=\{\xi_1, \dots, \xi_N\}$ be a maximal $(\da,\rho_\Og)$-separated subset of $U(\xi, L \da)$.
	It is enough to show that $N\leq C L^d$. If $\dist(\xi, \Ga) \ge 4(L\da)^2$, then $\dist(\xi_j, \Ga)\sim \dist(\xi,\Ga) $ for $1\leq j\leq N$, which, by  \eqref{3-17}, implies that  $$|U(\xi, L\da)|\sim (L\da)^{d}\sqrt{\dist(\xi, \Ga)}\   \ \text{and}\   \ |U(\xi_j,\da/2)|\sim \da^{d} \sqrt{\dist(\xi, \Ga)}.$$ Thus, a standard  volume comparison argument shows that $N\leq C L^d$ if  $\dist(\xi, \Ga) \ge 4(L\da)^2$.

	Now  assume that $\dist(\xi,\Ga) < 4(L\da)^2$. Then  $\dist(\xi_j,\Ga) < 9(L\da)^2$ for $1\leq j\leq N$. Let $m_0\in\NN$ be such that $2^{m_0-1} \leq 3L<2^{m_0}$. Then 
	$\Ld\subset \bigcup_{k=0}^{m_0} A_k$, where 
	\begin{align*}
	A_0&=\Bl\{ \eta\in\Og:\  \ \|\eta-\xi\|\leq  L \da,\   \    \dist( \eta,\Ga) \leq \da^2\Br\},\\
	A_k:&=\Bl\{ \eta\in\Og:\  \ \|\eta-\xi\|\leq  L \da,\   \   ( 2^{k-1} \da)^2< \dist( \eta,\Ga) \leq ( 2^k \da)^2\Br\},\   \ k\ge 1.
	\end{align*}
By  \eqref{3-17} and \eqref{3-20}, we have that for $0\leq k\leq k_0$,  \begin{align*}
	|A_k|&\leq C (L\da)^{d-1} 4^k \da^2\   \   \ \text{and}\   \ 
	|U(\eta,\da/2)| \sim 2^k \da^{d+1},\   \   \  \forall \eta\in A_k.
	\end{align*}
	Thus, a standard  volume comparison argument shows that 
	\begin{align*}
N_k:=	\# (\Ld \cap A_K) \leq C 2^k L^{d-1},\   \ k=0,1,\dots, m_0.
	\end{align*}
	It follows that 
	 $N\leq \sum_{k=0}^{m_0} N_k\leq C L^d.$

	(iii) Let $\k_0>0$ be the parameter of the $C^2$-domain $\Og$ given in Definition \ref{def-C2}. Then there exists  a point  $\xi\in \Og$ such that $\dist(\xi, \p \Og) >\f 1 2 \k_0$.  By the definition of the metric  $\rho_\Og$ (see \eqref{metric}), we have $\Og=U(\xi, L\da)$, where 
	$$L=\f{\Bl( 1+ \f 2 {\k_0}\Br) \diam(\Og)}\da.$$
	Thus,  every $\da$-separated subset $\Ld$ of $\Og$ is contained in a maximal $\da$-separated subset of $U(\xi, L\da)$. By the proof in Part (ii), it then follows that
	\[ \# \Ld \leq C_\Og L^d\leq C_{\Og}' \da^{-d}. \]
	
\end{proof}

\section{Bernstein inequality on domains of special type in $\RR^2$ }\label{sec:12}


In this section, we shall prove  a Bernstein type inequality on domains of special type in $\RR^2$. 
Let $G\subset \RR^2$ be an $x_2$-upward domain given by 
\begin{equation}\label{4.1-10-2} G:=\Bl\{ (x, y):\  \ x\in (-a,a),\   \  g(x)-L a< y\leq g(x)\Br\},\end{equation}
where  $a>0$, $L\ge 1$  and $g$ is a $C^2$-function on $\RR$ satisfying  $\min_{x\in [-2a, 2a]} g(x)= 4La$.
Following the notations in Section ~\ref{decom-lem}, we have that
\begin{align*}
G^\ast:&=\Bl\{( x,   y):  x\in (-2a,2a),   0 <  y \leq g(x)\Br\},
\end{align*}
and  for each  $\mu>0$, 
\begin{align}
	G(\mu):&=\{ (x, y):\  \ x\in (-\mu a,\mu a),\   \  g(x)-\mu L a<y\leq g(x)\},\     \label{4.2-10-8}\\
	\p' G (\mu) :&=\{ (x, g(x)):\  \   x\in (-\mu a, \mu a)\},\    \   \ \p' G=\p' G(1),\  \ \p'G^\ast=\p'G(2).\notag
\end{align}  
For $(x,y)\in G^\ast$,  we define  
\begin{equation*}
	\da(x,y):=g(x)-y\   \   \text{and}\  \   \
	\vi_n(x,y) :=\sqrt{\da(x,y)} +\f 1n,\   \  n=1,2,\dots.
\end{equation*}
According to Lemma~\ref{lem-9-1}, we have 
\begin{equation*}\label{10-6}
	\da(x,y)=g(x)-y \sim \dist (\xi, \p' G^\ast),\   \    \  \forall\xi= (x,y)\in G.
\end{equation*}

The  Bernstein type inequality on the domain  $G$ is formulated in terms of  the tangential derivatives along the essential boundary $\p' G$ of $G$. To be precise,  for each   fixed  $x_0\in (-2a, 2a)$ and  positive integer $\ell$, we denote by ${\mathcal{D}}_{x_0}^\ell$ the $\ell$-th order directional derivative along the    tangential direction $(1, g'(x_0))$ to  $\p' G^\ast$ at the point $(x_0, g(x_0))$: 
\begin{equation*}
	{\mathcal{D}}_{x_0}^\ell: =\Bigl (\p_1 +g'(x_0)\p_2\Bigr)^{\ell}= \sum_{i=0}^\ell \binom{\ell}{i} (g'(x_0))^{i} \p_1^{\ell-i} \p_2^i,
\end{equation*}
where $\p_1=\f {\p}{\p x}$ and $\p_2=\f {\p}{\p y}$.
With a slight abuse of  notation,  we denote by  ${\mathcal{D}}^{(\ell)}$  the $\ell$-th order tangential differential operator  given  by  
\begin{equation}\label{tagential} {\mathcal{D}}^{(\ell)} f(x,y):=({\mathcal{D}}_x^\ell f)(x,y)=\sum_{i=0}^\ell \binom{\ell}{i} (g'(x))^{i} (\p_1^{\ell-i} \p_2^i f)(x,y), \end{equation}
where $f\in C^1(G^\ast)$ and  $(x,y)\in G^\ast$.
Clearly,    the operator ${\mathcal{D}}^{(\ell)}$ is commutative with  $\p_2$, but not commutative with $\p_1$.  
As a result, the operators $\mathcal{D}^{(\ell)}$, $\ell=1,2,\cdots$ are not commutative. 
It is also worthwhile to  point out  here that   $\mathcal{D}^{(\ell)}$ is not the $\ell$-th power $\mathcal{D}^{\ell}$  of the operator $\mathcal{D}$; namely,  $ \mathcal{D}^{(\ell+1)}\neq \mathcal{D}^{(\ell)} \mathcal{D}$.

Our aim in this section is to prove the following Bernstein type  inequality on the domain  $G$:

\begin{thm}\label{THM:2D BERN}  Let $G\subset \RR^2$ be the  domain given in \eqref{4.1-10-2}. 
	If   $0<p\leq \infty$, $\ld >1$ and      $f\in \Pi_n^2$, then 
	\begin{equation}\label{10-7-18}
		\|\vi_n^{i}{\mathcal{D}}^{(r)}\partial_2^{i+j} f\|_{L^p(G)}\le  c n^{r+i+2j}\|f\|_{L^p(G(\ld))},\  \  r,i,j=0,1,\dots, 
	\end{equation}
	where   $c$ is a positive constant independent of $f$ and $n$, and the set $G(\ld)$ is given in \eqref{4.2-10-8}.
\end{thm}

Theorem ~\ref{THM:2D BERN} allows us to establish a   slightly stronger Bernstein type inequality on the domain $G$. 

\begin{cor}\label{cor-10-2} Let $\ld>1$ and $\mu>1$ be two given parameters.  Let $G\subset \RR^2$ be the   domain given in \eqref{4.1-10-2}.  Then for any   $0<p\leq \infty$  and       $f\in\Pi_n^2$,   
	\begin{align}
		\Bl\|\vi_n(x,y)^{i} &\max_{|t|\leq \min\{ \mu \vi_n(x,y), a\}} \Bigl|  {\mathcal{D}}_{x+t}^{r}\partial_2^{i+j}f(x,y)\Br|\Br\|_{L^p(G)}\notag\\
		&\le  c\mu^r  n^{r+2j+i}\|f\|_{L^p(G(\ld))},\  \  r,i,j=0,1,\dots,\label{4-7a}
	\end{align}
	where $c>0$ is independent of $f$, $n$ and $\mu$, the first $L^p$-norm is computed with respect to the Lebesgue measure $dxdy$ on the domain $G$, and  the set $G(\ld)$ is given in \eqref{4.2-10-8}.
\end{cor}

\begin{rem}\begin{enumerate}[\rm (i)]
		\item It can be easily shown that the order $n^{r+i+2j}$ in   \eqref{10-7-18} and \eqref{4-7a}  is sharp as $n\to\infty$. 
		\item 	Using translations and reflections, similar results can also be established  on any domain $G\subset \RR^2$ of special type as defined in Definition~\ref{def:specialtype} without the further technical assumptions made in this section.
	\end{enumerate}
 
\end{rem}

The rest of this section is organized as follows. In subsection ~\ref{subsec-4:1},  we prove Corollary~\ref{cor-10-2}, assuming 
  Theorem~\ref{THM:2D BERN}.   The proof of   Theorem~\ref{THM:2D BERN} for   $r=1$, which is relatively simpler, but already  contains the  crucial  ideas required for the proof for $r>1$, 
   is given 
  in subsection~\ref{subsubsection:10-1}.  Finally, we prove Theorem~\ref{THM:2D BERN} for  $r>1$ in subsection~\ref{subsubsection:10-2}, which is technically more involved.

\subsection{Proof of  Corollary~\ref{cor-10-2} (assuming Theorem~\ref{THM:2D BERN})}\label{subsec-4:1}
Let  $M>1$  denote a constant such that 
\begin{equation}\label{10-7}
\|g''\|_{L^\infty([-2 a,2 a])}\le M\  \  \text{and}\  \  |g' (0)|\le M.
\end{equation}
	For any fixed   $x_0, t\in [-a,a]$, we have 
	\begin{equation}\label{10-8-eq} {\mathcal{D}}_{x_0+t} ={\mathcal{D}}_{x_0}+\Bl( g'(x_0+t)-g'(x_0)\Br) \p_2= {\mathcal{D}}_{x_0} +|t|O(x_0,t)  \p_2,\end{equation}
	where $\sup_{ t\in [-a,a]} |O(x_0,t)|\leq M$.
	Since  the operators $\p_2$ and ${\mathcal{D}}_{x_0}^i$ are commutative,   we obtain from~\eqref{10-8-eq} that if $(x_0,y_0)\in G$  and $|t|\leq \min\{ \mu \vi_n(x_0,y_0), a\}$, then 
	\begin{align*}
		\Bl|\vi_n(x_0,y_0)^i {\mathcal{D}}_{x_0+t} ^r \p_2 ^{i+j}  f(x_0,y_0)\Br| &\leq C \mu^r \max_{r_1+r_2=r} \vi_n(x_0,y_0)^{i+r_2} |{\mathcal{D}}_{x_0}^{r_1} \p_2^{i+j+r_2} f(x_0,y_0)|.
	\end{align*}
	It then follows from Theorem~\ref{THM:2D BERN} that 
	\begin{align*}
		\Bl\|&\vi_n(x,y)^{i} \max_{|t|\leq \min\{ \mu \vi_n(x,y), a\}} \Bigl| {\mathcal{D}}_{x+t}^{r}\partial_2^{i+j}f(x,y)\Br|\Br\|_{L^p(G; dxdy)}\\
		&\leq C \mu^r \max_{r_1+r_2=r} \Bl\| \vi_n^{i+r_2} {\mathcal{D}}^{(r_1)} \p_2 ^{i+r_2+j} f\|_{L^p(G)}\leq C\mu^r n^{i+r+2j}\|f\|_{L^p(G(\ld))}.
	\end{align*}
	This completes the proof of Corollary~\ref{cor-10-2}.

\subsection{Proof of Theorem~\ref{THM:2D BERN} for $r=1$} \label{subsubsection:10-1}

This subsection is devoted to the proof of Theorem~\ref{THM:2D BERN} for $r=1$.  Since the term on the right hand side of \eqref{10-7-18} is increasing in $\ld>1$, without loss of generality, we may assume that $\ld\in (1,2)$. 
We need two lemmas, the first of which    is a direct  consequence of the univariate Bernstein inequality for algebraic polynomials.
\begin{lem}\label{lem:univ BM} Let $G\subset \RR^2$ be the   domain given in \eqref{4.1-10-2}.
	Assume that $f\in C(G)$ satisfies that  $f(x, \cdot)\in\Pi_n^1$  for each fixed  $x\in [-a,a]$. Then for $0<p\leq \infty$ and  $\ld\in (1,2)$, 
	\[
	\|\vi_n^i\partial_2^{i+j} f\|_{L^p(G)}\le  c n^{i+2j}\|f\|_{L^p(G(\ld))},\   \  i,j=0,1,\dots, 
	\]
	where 
	$c>0$ is a constant independent of $f$ and $n$. 
\end{lem}
\begin{proof}
	If $P$ is an algebraic  polynomial of one variable of degree $\le n$, then by the   univariate Bernstein inequality  (\cite[p. 265]{De-Lo}),  we have that for any $b>0$ and $\al>1$, 
	\begin{equation}\label{markov-bern}
		\Bl	\|(\sqrt{b^{-1}t} +n^{-1})^iP^{(i+j)}(t)\Br\|_{L^p([0,b], dt)} \le  C_\al  n^{i+2j}  b^{-(i+j)}
		\|P\|_{L^p([0,\al b])}.
	\end{equation}
	Lemma~\ref{lem:univ BM} then follows  by integration over vertical line segments. 	
\end{proof}


Our next lemma, Lemma ~\ref{lem-1-1:Berns} below,  plays a crucial  role in the proof of Theorem~\ref{THM:2D BERN}. We first explain briefly  the basic idea behind this lemma.   Note that  every point  $(x,y)\in G$ can be connected with  the point $(x, g(x))\in  \p' G$ via  a vertical line segment.
 Taking into account the tangential direction along the boundary of $G$,   we shall  construct a family   of parabolas which  touch and lie below the essential boundary $\p' G$, and use them to replace these vertical line segments.  Lemma ~\ref{lem-1-1:Berns} below  asserts   that   each  point $(x,y)\in G$ can  be connected with a point $(z, g(z))$ on $\p' G^\ast$ via   a unique   parabola    which  passes   though    $(x,y)$ and  touches $\p' G^\ast$ at $(z, g(z))$.  As a result, performing a change of variables, we may express    double  integrals over the domain  $G$ in terms of   iterated integrals along the family of parabolas, and after that   apply the univariate Bernstein inequality to  integrals along these  parabolas.

To be more precise,  let  $M> 10$  be  any fixed constant  satisfying~\eqref{10-7}.
Given a parameter  $A>0$, we   define 
$$ Q_A(z,t):=g(z)+g'(z) t -\f A2t^2,  \   \  z\in [-2 a, 2 a],\  \ t\in\RR. $$ 
By  Taylor's theorem,  if $z, z+t\in [-2 a,  2a]$, then 
\begin{equation}\label{eqn:taylor} 
	g(z+t)-Q_A(z,t) =\int_z^{z+t} [ A +g''(u)] (z+t-u)\, du.\end{equation}
Thus, for each fixed $z\in [-2a, 2a]$ and each parameter $A>M$,  $$y=Q_A(z, x-z)=g(z)+g'(z) (x-z) -\f A2 (x-z)^2,\   \ x\in [-2a,2a]$$  is a parabola that lies below the curve $y=g(x)$, but  touches it  at the point $(z, g(z))\in~\overline{ \p' G^\ast}$.

 Set
$\ld:=1+\f 1{M}$. 
Let  
$$ E_A:=\Bl\{(z,t)\in\RR^2:\   \  z, z+t\in [-\lambda a,\lambda a],\   \   |t|\leq a_0:=\sqrt{\frac{2La\lambda}{A+M}}\Br \},$$
and define the  mapping $\Phi_A: E_A\to \RR^2$ by
\begin{equation*}\label{1-2:Phi}
\Phi_A(z, t)=(x,y): =\big(z+t, Q_A(z,t)\big),\   \     (z,t) \in E_A.
\end{equation*}
We set  $E_A^{+} =\{(z,t)\in E_A:\  \ t\ge 0\}$,  $E_A^{-} =\{(z,t)\in E_A:\ \ t\leq 0\}$, and 
denote by $\Phi_A^+$ and $\Phi_A^{-}$ the restrictions of $\Phi_A$ on the sets $E_A^+$ and $E_A^{-}$ respectively. Also, we denote by $J_{\Phi_A}$ the Jacobian of the mapping $\Phi_A$.

With the above notation, we have the following crucial  lemma:

\begin{lem}\label{lem-1-1:Berns}  Let $\overline{A}:=~(2+\f {16L}a)M^2+M$. Then the following statements hold for every parameter  $A\ge\overline{A}$: 
	\begin{enumerate}[\rm (i)]
		
			\item   $ \Phi_A (E_A) \subset G(\ld)$, and  moreover,  both the  mappings $\Phi_A^+: E_A^{+}\to G(\ld)$ and $\Phi_A^{-} : E_A^{-} \to G(\ld)$ are   injective.  
			\item For every $(x,y)\in G$, there exists a unique $z\in [-\ld a, \ld a]$ such that $z\leq x\leq z+a_1$, and both $(x,y)$ and $(z, g(z))$ lie on the same parabola 
			$$\Bl\{ (u,v):\  \ v=Q_A(z, u-z),\   \ z\leq u\leq z+a_1\Br\},$$
			where  $$ a_1:=\sqrt{\frac{2La}{A-M}}<a_0. $$ 
			As a result, 		every  $(x,y)\in G$ can be represented uniquely in the form  
		$(x,y)=\Phi_A(z,t)$ with  $(z,t)\in E^{+}_A$  and $0\leq t\leq a_1$, and moreover,  $G\subset \Phi_A(E_A^{+})$.
	
		\item If  $(z,t)\in E_A$, then    
		\begin{equation}\label{1-3:jacobian}
			\bigl| \det  \left(J_{\Phi_A}(z,t)\right) \bigr| = 
			(A+g''(z)) |t|,
		\end{equation}
	where we recall that $J_{\Phi_A}$ denotes  the Jacobian of the mapping $\Phi_A$.

		\item 	Let    $u_A$  be the  function on the set  $\Phi_A(E_A^{+})$ given by 
		\begin{equation}\label{10-12}
			u_A (x,y): =g'(z+t) -g'(z) +A t=:w_A(z,t),
		\end{equation}
		where   $(x,y)=\Phi_A(z,t)$  and   $(z, t)\in E_A^{+}$.
		Then  for every $(x,y)\in \Phi_A (E_A^{+})$, 
		\begin{equation}\label{1-6:bern}
			\frac{(A-M)\sqrt{2}}{\sqrt{A+M}} \, \sqrt{\delta(x,y)} \le u_A(x,y)\le \frac{(A+M)\sqrt{2}}{\sqrt{A-M}} \, \sqrt{\delta(x,y)}.
		\end{equation}
	\end{enumerate}
\end{lem}

\begin{proof} 
	
	(i) First, we prove  that $\Phi_A (E_A) \subset G(\ld).$ Indeed, if  $(x,y) =\Phi_A(z,t) $ with $(z, t) \in E_A$, then  $x=z+t\in [-a\lambda,a\lambda]$ and by~\eqref{eqn:taylor},  
	\begin{equation} \label{1-4:taylor} 
		g(x)-y=g(z+t)-Q_A(z,t) =\int_z^{z+t} (g''(u)+A) (z+t-u)\, du.
	\end{equation} 
	Since $|t|\leq a_0$, this implies  that 
	$0\leq g(x)-y\leq (M+A) a_0^2/2 =\lambda L a$, which shows that  $(x, y)\in G(\ld)$.
	
	Next, we show that both of the mappings $\Phi_A^{+}$ and $\Phi_A^{-}$ are injective. 	  Assume that $\Phi_A(z_1, t_1)=\Phi_A(z_2, t_2)$ for some  $(z_1, t_1), (z_2, t_2) \in E_A$ with $t_1t_2\ge 0$ and $t_2\ge t_1$.  Then $z_1+t_1=z_2+t_2=:\bar{x}$, $Q_A(z_1,t_1)=Q_A(z_2,t_2)$, and hence, using~\eqref{1-4:taylor}, we obtain that for $i=1,2$,
	\begin{align*}
		g(\bar{x}) -Q_A(z_i,t_i) =g(z_i+t_i)-Q_A(z_i,t_i) =\int_0^{t_i} [ g''(\bar{x}-v) +A] v \, dv,
	\end{align*}
	which implies 
	\begin{equation}\label{1-5:inj}
		\int_{t_1}^{t_2} \Bl( g''(\bar{x} -v)+A\Br) v\, dv =0.
	\end{equation}
	Since $g''(\bar{x} -v)+A\ge A-M>0$ and $v$ doesn't change sign on the interval $[t_1, t_2]$,~\eqref{1-5:inj} implies that  $t_1=t_2$, which  in turn implies that $z_1=z_2$. Thus, both  $\Phi_A^+$ and $\Phi_A^{-}$ are  injective.

	(ii) First, a straightforward calculation shows that for $A\ge \overline{A}$, 
	\begin{equation*}
	 a_1<a_0\    \ \text{and}\     \  a+a_1<\ld a.
	\end{equation*}

	Next,  we show that there exists $z\in [-\ld a, \ld a]$ such that $y=Q_A(z, x-z)$ and $z\leq x\leq z+a_1$. 
	Define  
	$h(s):= g(x) -Q_A(x-s,s)$ for $0\leq s\leq a_1$. 
	 Using~\eqref{eqn:taylor} with $z=x-s$, we obtain 
	\[
	h(s)=\int_{x-s} ^x (g''(u) +A) (x-u)\, du,
	\]
which  implies that 
	\[
	h(a_1) \ge (A-M) \int_0^{a_1} v\, dv= \f 12 (A-M) a_1^2=L a.
	\] 
	Since  $h$ is a continuous function on $[0, a_1]$ and  $h(0)=0$, it follows by  the Intermediate Value theorem that 	$[0, La]\subset h[0, a_1]$.
	Since  $(x,y) \in G$, we have  $-a\leq x\leq a$, and   $g(x)-y\in [0, La]$. Thus,  there exists $t\in [0,a_1]$ such that $h(t)=g(x)-y$. Setting   $z=x-t$, we have 
	$y=Q_A(z, x-z)$ and 
		$$-\ld a\leq -a-a_1\leq z= x-t \leq x\leq a.$$

	Finally, we show the uniqueness of the number $z$. Indeed, setting $t=x-z$, we have 
	$$(x,y)=\Phi_A^{+} (z,t)=(z+t, Q_A(z,t)),\    \    (z,t)\in E_A,\   \ 0\leq t\leq a_1.$$
	The uniqueness of the number $z$ then follows from the fact that the mapping   $\Phi_A^{+}$ is injective.

		(iii) Equation~\eqref{1-3:jacobian}  can be verified  straightforwardly.

	(iv) Let $(x,y)=\Phi_A (z,t)$ with $(z,t)\in E_A^{+}$. 
	By~\eqref{10-12} and the mean value theorem, we have 
	\begin{equation}\label{10-16}
		(A-M)t\le u_A(x,y)\le (A+M)t.
	\end{equation}
	On the other hand, however,  using~\eqref{eqn:taylor}, we have that 
	\begin{equation*}
		g(x) -y =g(z+t) -Q_A(z,t) =\int_z^{z+t} ( g''(u) +A) (z+t-u) \, du,
	\end{equation*}
	which in particular  implies 
	\begin{equation}\label{10-17}
		(A-M)\tfrac{t^2}{2}\le g(x)-y\le (A+M)\tfrac{t^2}{2}.
	\end{equation}
	Finally, combining~\eqref{10-16} with~\eqref{10-17}, we deduce the desired estimates 
	\eqref{1-6:bern}.
\end{proof}

\begin{rem}
	Lemma~\ref{lem-1-1:Berns} with  several different parameters $A$  will be   required in our proof of Theorem~\ref{THM:2D BERN} for $r>1$.   However,    for the proof in  the case of $r=1$,  it will be enough to have this lemma  for  one  fixed parameter $A\ge \overline{A}$ only. 
\end{rem}

We are now in a position to prove  Theorem~\ref{THM:2D BERN} for $r=1$.

\begin{proof}[Proof of  Theorem~\ref{THM:2D BERN} for $r=1$]
	If   $f\in\Pi_n^2$ and $x\in [-2a, 2a]$, then  ${\mathcal{D}}_x f(x,\cdot)$ is a polynomial  of degree at most $n$ of a single variable.  
	Since  the operators $\mathcal{D}$ and $\p_2$ are commutative,  it follows from Lemma~\ref{lem:univ BM} that 	
	$$\|\vi_n^{i}\mathcal{D}\partial_2^{i+j} f\|_{L^p(G)}=\|\vi_n^{i}\partial_2^{i+j} \mathcal{D} f\|_{L^p(G)}\leq C n^{i+2j} \|\mathcal{D} f\|_{L^p(G(\sqrt{\ld}))}.$$
	Thus,  it is sufficient  to show that for any $\ld\in (1,2)$, 	\begin{equation}\label{10-19:es}
		\|\mathcal{D}f\|_{L^p(G)} \le C n \|f\|_{L^p(G(\ld))},\  \  \forall f\in\Pi_n^2, 
	\end{equation} 
	  Without loss of generality, we may assume that   $M:=\f 1{\ld-1}$ satisfies~\eqref{10-7}.  
	Also, in our proof below, we shall always  assume that $p<\infty$.  The  case $p=\infty$ can be treated similarly, and in fact, is simpler.

	To prove~\eqref{10-19:es}, we set 
	$$	I:=\|\mathcal{D}f\|_{L^p(G)}^p=	\iint_{G} \Bl|{\mathcal{D}}_x f(x,y)|^p\, dxdy,$$
	and let  $A\ge \overline{A}$ be a fixed parameter. Using   Lemma~\ref{lem-1-1:Berns} (i) and (iii), and  performing the  change of variables $x=z+t$ and $y=Q_A(z, t)$,  we   obtain 
	\begin{align*}
		I& =\iint_{(\Phi_A^{+})^{-1} (G)} \Bl| {\mathcal{D}}_{z+t} f(\Phi_A(z,t))\Br|^p (A +g''(z)) t \, dzdt.
	\end{align*}
	A straightforward calculation shows that  for each $(z,t)\in E_A$,
	\begin{equation*}\label{eqn:part_der} 
		{\mathcal{D}}_{z+t} f( \Phi_A (z,t))=\f d{dt} \Bl[ f( \Phi_A(z,t))\Br] +w_A(z,t)\p_2 f (\Phi_A(z,t)),
	\end{equation*}
	where $w_A(z,t)= g'(z+t)-g'(z) +At$. 
	Thus,     $I\leq 2^p  (I_1+I_2)$, where 
	\begin{align*}
		I_1&:=\iint_{(\Phi_A^+)^{-1} (G)} \Bl| \f d {dt} \Bl[ f(\Phi_A(z,t))\Br]\Br|^p (A +g''(z)) t\, dzdt,\\
		I_2&:=  \iint_{(\Phi_A^+)^{-1} (G)} |w_A(z,t)|^p |\p_2 f(\Phi_A(z,t))|^p (A +g''(z)) t \, dzdt.\notag
	\end{align*}
	
	For  the  double integral $I_2$,   performing the change of variables $(x,y)=\Phi_A^{+} (z,t)$, and 
	using    Lemma~\ref{lem-1-1:Berns}, we obtain   
	\begin{align}\label{1-7:Bern}
		I_2& =\iint_G |u_A(x,y)|^p |\p_2 f(x,y)|^p\, dxdy\leq C  \iint_G  |\sqrt{\da(x,y)}\p_2 f(x,y)|^p\, dxdy\notag\\
		&=C\int_{-a}^a \int_{g(x)-La}^{g(x)} |\sqrt{\da(x,y)}\p_2 f(x,y)|^p\, dydx,\notag
	\end{align}
	which, using
	the Markov-Bernstein-type  inequality~\eqref{markov-bern}, is bounded above by 
	\begin{align*}
		\leq 	C n^p \int_{-a}^a  \int_{g(x)-\ld L a}^{g(x)} |f(x,y)|^p\, dydx\leq C n^p \|f\|_{L^p(G(\ld))}^p.
	\end{align*}
	
	Thus, it remains to prove  that \begin{equation}\label{10-23-0}
		I_1	\leq C_p n^p \|f\|_{L^p(G(\ld))}^p.
	\end{equation}
	Since $A\ge  \overline{A}$, we have $a_0 <\f {a} {2M}$. Thus, 
	using  Lemma~\ref{lem-1-1:Berns} (ii),  we obtain  that in the $zt$-plane, 
	\begin{equation}\label{10-24-0}
		(\Phi_{A}^{+})^{-1} (G) \subset [-a-a_1,a]\times [0, a_1]\subset [-a-a_0, a] \times [-a_0, a_0] \subset E_{A}.
	\end{equation}	
It follows that  
	\begin{align*}
		I_1		&\leq C(A) \int_{- a-a_1 }^{a} \Bl[ \int_{0}^{a_1}\Bl| \f d{dt} \Bl[ f(z+t, Q_A(z, t))\Br] \Br|^p |t|\, dt\Br]\, dz,\end{align*}
		which,  using Bernstein's inequality with doubling weights  (\cite[Theorem~7.3]{MT2}, \cite[Theorem~3.1]{Er}) and the fact that $a_1<a_0$, is bounded above by 
		\begin{align*}
	&\leq C(M,p) n^p  \int_{-a-a_1 }^{a} \Bl[ \int_{-a_0}^{a_0}\Bl| f(z+t, Q_A(z, t)) \Br|^p |t|\, dt\Br]\, dz\\
		&	\leq C(M,p)  n^p \iint_{E_A} |f(z+t, Q_A(z,t))|^p|t|\, dzdt. \end{align*}
	Splitting this last double integral into two parts $\iint_{E_A^{+}}+\iint_{E_A^{-}}$, and applying  the change of variables $(x,y)=\Phi_A(z,t)$ to each of them  separately, we obtain that 
	\begin{align*}
		I_1&	\leq   C(M,p)n^p \iint_{E_A^{+}\cup E_A^{-}} |f(\Phi_A(z,t))|^p(A +g''(z)) |t|\, dzdt \\
		&\leq C(M,p) n^p \Bl[ \iint_{\Phi(E_A^{+})}  |f(x,y)|^p dxdy+\iint_{\Phi(E_A^{-})}  |f(x,y)|^p dxdy \Br]\\ &\leq C(M,p) n^p \|f\|_{L^p(G(\ld))}^p,	 
	\end{align*}
	where we used Lemma~\ref{lem-1-1:Berns}~ (i), (iii) in the second step,  and  the fact that $\Phi_A (E_A) \subset G(\ld)$ (i.e.,  Lemma~\ref{lem-1-1:Berns}~ (i)) in the third step. 
\end{proof} 

\subsection{Proof of Theorem~\ref{THM:2D BERN} for $r>1$}\label{subsubsection:10-2} Again, without loss of generailty, we assume $\ld\in (1,2)$.  First, we note that by \eqref{tagential} and the comments immediately after \eqref{tagential},  the operators $\mathcal{D}^{(r)}$, $r=1,2,\dots$ are not commutative. As a result,  we cannot deduce the Bernstein inequality \eqref{10-7-18}  for $r>1$ directly  from the already proven case of $r=1$  via iteration.  
The proof for $r>1$ requires   a few additional lemmas. 

%

We start with the following   lemma for  computing  higher order derivatives of certain composite functions: 
\begin{lem}\label{lem:mult faa di bruno}
	Let $f\in C^\infty (\R^2)$ and  define  $F(t):= f(z+t,Q(t))$ for  $t\in\R$, where $z\in\RR$ and $Q$  is a univariate  polynomial of degree at most $2$.  Then for any $r\in\NN$,
	\begin{align*}
	\frac{d^rF}{dt^r} =  \sum_{i+j+2k=r} \frac{r!}{i!j!k!2^k}(Q')^j(Q'')^k (\partial_1^{i} \partial_2^{j+k}  f) + \sum_{2(i+j)=r} \frac{r!}{i!j!2^{r/2}} (Q'')^j (\partial_1^{i} \partial_2^j  f), 
	\end{align*}
	where the summations are taken over non-negative integers $i$, $j$, $k$, and the partial derivatives are evaluated at $(z+t,Q(t))$.
\end{lem}
\begin{proof}
	This is a special  case of the multivariate F\'aa di Bruno formula from~\cite[p.~ 505]{Co-Sa}.
\end{proof}

Lemma \ref{lem:mult faa di bruno} allows us to compute the higher order derivatives of the composite function $f\circ \Phi_A$:
\begin{lem}\label{lem:bern repres}  Let $A\ge \overline{A}$ be an arbitrarily given parameter, and let   $(z_0, t_0) \in E_A^{+} $ be such that $(x_0, y_0)
	=\Phi_A (z_0, t_0)\in G$.
Assume that  $f\in C^\infty (\RR^2)$ and   $F(t):=f(\Phi_A(z_0, t))$ for  $t\in\R$. Then for any $r\in\NN$,
	\begin{align}\label{eqn:f_s1_s2_s3}
	\frac{d^rF(t_0)}{dt^r}
	= S_1(A,x_0,y_0)+S_2(A,x_0,y_0)+S_3(A,x_0,y_0),
	\end{align}
	where
	\begin{align}
	S_1(A,x_0,y_0) &:=  \sum_{j=0}^r \binom{r} j (-u_A(x_0, y_0))^j ({\mathcal{D}}^{(r-j)} \partial_2^{j} f)(x_0,y_0),\label{s1}\\
	S_2(A,x_0,y_0) &:=  \sum_{\sub{i+j+2k=r\\
			k\ge1}} \frac{r!}{i!j!k!2^k}(-u_A(x_0, y_0))^j(-A)^k ({\mathcal{D}}^{(i)} \partial_2^{j+k} f) (x_0,y_0),\label{s2} \\
	S_3(A,x_0,y_0) &:=  \sum_{2(i+j)=r} \frac{r!}{i!j!2^{r/2}} (-A)^j ({\mathcal{D}}^{(i)} \partial_2^j f) (x_0,y_0).\label{s3}
	\end{align}
\end{lem}

\begin{proof}
	Define  $\wt f(u,v):=f(u,g'(x_0)u+v) $ for  $ (u,v)\in\RR^2$.
	Then 
	\begin{equation}\label{4-29}
	\partial^i_1 \wt f(u,v) = {\mathcal{D}}^i_{x_0}f\bl(u, g'(x_0) u+v\br)\   \  \text{and}\   \  \p_2^j\wt f(u,v)=\partial_2^j f\bl(u, g'(x_0) u +v\br).
	\end{equation}
	
	We  rewrite  the function $F$ as 
	$$
	F(t)=f\bl(z_0+t, Q_A(z_0, t)\br) = \wt f\bl(z_0+t, Q(t)\br),$$
	where 
	$$ Q(t):= Q_A (z_0, t)-g'(x_0)(z_0+t)= g(z_0) +g'(z_0) t -\f A 2 t^2 -g'(z_0+t_0) (z_0+t).$$
	Then \eqref{4-29} implies that 
	\begin{align*}
	\partial^i_1 \wt f(z_0+t_0, Q(t_0))& = {\mathcal{D}}^i_{x_0}f\bl(x_0, y_0\br)\   \  \text{and}\   \  \p_2^j\wt f(z_0+t_0, Q(t_0))=\partial_2^j f\bl(x_0, y_0\br).
	\end{align*}
	To  complete the proof of   ~\eqref{eqn:f_s1_s2_s3}, we just need to apply  Lemma~\ref{lem:mult faa di bruno} to the function $F(t)$,   observing   that 
	$$ Q'(t_0) = g'(z_0) -A t_0 -g'(z_0+t_0) =-w_A(z_0, t_0)=-u_A(x_0, y_0),$$
	and 
	$Q''(t)=-A$,
	
\end{proof}

The following lemma allows us to use   Lemma~\ref{lem-1-1:Berns} for   several distinct  parameters~ $A$.
\begin{lem}\label{lem:A_0 A_r} Given a positive integer $r$, there exist  constants $A_0, A_1, \dots, A_r$ depending only on $M$ and $r$ such that  $\overline{A}\leq A_0<\dots<A_r$ and  for any $(x_0,y_0)\in G$, and  $f\in C^r(\RR^2)$,
	\begin{equation}\label{eqn:a_v_bound}
	\max_{0\le i\le r}\left|(\delta(x_0,y_0))^{i/2}({\mathcal{D}}^{(r-i)}\partial_2^if)(x_0,y_0)\right| \le c \max_{0\le v\le r} \left|S_1(A_v,x_0,y_0)\right|,
	\end{equation}
	where the constant $c$ depends only on $M$,  $r$ and  $a$, and $	S_1(A,x_0,y_0)$ is given in~\eqref{s1}. 
\end{lem}

\begin{proof}
	Note first   that  if $\da(x_0, y_0)=0$, then  by \eqref{1-6:bern}, $u_A(x_0, y_0)=0$, and hence by \eqref{s1}, 
	$S_1(A, x_0, y_0) =\mathcal{D}^{(r)} f(x_0, y_0)$.
	Thus,  \eqref{eqn:a_v_bound} holds trivially if $\da(x_0, y_0)=0$. For the reminder  of the proof, we always assume that  $\da(x_0, y_0) >0$.

	Next, select  a strictly increasing sequence of  constants $A_0, A_1,\dots, A_r\ge \overline{A}$ depending only on $M$  and $r$  such that 
	for $i=0,1,\dots, r-1$, 
	$$\f {2 (A_i+M)}{\sqrt{A_i-M}} <\f {A_{i+1}-M} {\sqrt{A_{i+1}+M}},$$	which, using  Lemma~\ref{lem-1-1:Berns}~(iv), implies 
	\begin{equation}\label{4-31}
	 0<c_{M,r} \leq  \f{u_{A_i} (x_0,y_0)}{\sqrt{\da(x_0,y_0)}}<2\f{ u_{A_i} (x_0,y_0)}{\sqrt{\da(x_0,y_0)}}<\f{u_{A_{i+1}}(x_0,y_0)}{\sqrt{\da(x_0,y_0)}} \leq C_{M,r}.
	\end{equation}	
	Setting 
	\begin{align*}
	B_j = -\f { u_{A_j}(x_0, y_0)} {\sqrt{\da(x_0, y_0)}}\   \ \text{and}\   \  u_j:=  \binom{r} j\da(x_0, y_0)^{\f j2} {\mathcal{D}}^{(r-j)} \p_2 ^j f (x_0, y_0)
	\end{align*} 
	for $j=0,1,\dots, r$,  we may  apply  ~\eqref{s1}  to obtain  
	$$S_1 ( A_i, x_0, y_0) =\sum_{j=0}^r (B_i)^j u_j,\   \   \  i=0,1,\dots, r,$$
	which can be written equivalently as 
	\begin{align}\label{10-26-0}
	\mathbf{S} =\mathcal{M} \mathbf{u},
	\end{align}
	where 
	\begin{align*}
	\mathbf{S}&= \Bl (S_1 ( A_0, x_0, y_0), \dots, S_1 ( A_r, x_0, y_0)\Br)^t,\   \    \  
	\mathbf{u}= (u_0, u_1,\dots, u_r)^t,
	\end{align*} 
	and $\mathcal{M}$ is  the  $(r+1) \times (r+1)$ Vandermonde matrix 	with $(i,j)$-entry $\mathcal{M}_{i,j}:=(B_i)^j$. 
	By \eqref{4-31}, we have  
	$$ 0< c_{M} \leq \min_{0\leq i\neq j \leq r} |B_j-B_i|\leq C_{M}.$$
	Thus, the  stated estimate~\eqref{eqn:a_v_bound} follows from~\eqref{10-26-0} and  Cramer's rule.
\end{proof}

Now we are in a position to prove   Theorem~\ref{THM:2D BERN}
for $r>1$. 
\begin{proof}[Proof of Theorem~\ref{THM:2D BERN} for $r>1$.]
	We use induction on $r\in\NN$. 
	The case  $r=0$ is given in Lemma~\ref{lem:univ BM}, whereas the case $r=1$ has already been proven in subsection \ref{subsubsection:10-1}.  Now suppose the Bernstein type inequality \eqref{10-7-18}   has been  established for all the  derivatives ${\mathcal{D}}^{(\ell)} \p_2^{i+j}$ with  $\ell=0,1,\dots, r-1$ and  $i,j=0,1,\dots.$   Our goal is  to prove that for any $f\in \Pi_n^2$ and $\ld\in (1,2)$,
	\begin{equation*}\label{10-7-18b}
		\|\vi_n^{i}{\mathcal{D}}^{(r)}\partial_2^{i+j} f\|_{L^p(G)}\le  c n^{r+i+2j}\|f\|_{L^p(G(\ld))},\  \  i,j=0,1,\dots.
	\end{equation*}
	Since the operators ${\mathcal{D}}^{(i)}$ and $\p_2^j$ are commutative,  by Lemma \ref{lem:univ BM} ,   it is sufficient to show  that for every $f\in\Pi_n^2$ and any $\ld\in (1,2)$,
	\begin{align}\label{10-28}
		\|{\mathcal{D}}^{(r)} f\|_{L^p(G)}\le  c n^{r}\|f\|_{L^p(G(\ld))}.
	\end{align}
	Here and throughout the proof, the  constant $c$ depends only  on $p$, $a$, $M$ and $r$. 
	
We start with the case of $p=\infty$, which is simpler. By  Lemma~\ref{lem:A_0 A_r}, it is enough to show that 
	\begin{equation}\label{10-32-0}\max_{0\leq v\leq r} \max_{(x_0, y_0) \in G} |S_1(A_v,x_0,y_0)|\le c n^r\|f\|_{L^\infty(G(\ld))}. \end{equation}
	
	Fix temporarily $0\leq v\leq r$ and  $(x_0,y_0)\in G$. By Lemma~\ref{lem-1-1:Berns}~(ii),  there exists  $ (z_0, t_0)\in E_{A_v}^+$  such that  $0\leq t_0\leq a_1$ and $\Phi_{A_v}(z_{ 0},t_{0})=(x_0,y_0)$.  Applying  Lemma~\ref{lem:bern repres} to the function  $F(t):=f(\Phi_{A_v}(z_0,t))$, and 
	recalling (Lemma~\ref{lem-1-1:Berns}~(iv))
	\begin{align*}
		|u_{A_v}(x_0,y_0)|\le c \sqrt{\delta(x_0,y_0)},
	\end{align*}
	we  obtain
	\begin{align*}
		|S_1(A_v, x_0, y_0)|&\leq |F^{(r)} (t_0)| + |S_2(A_v, x_0, y_0)|+|S_3(A_v, x_0, y_0)|\\
		&\leq |F^{(r)} (t_0)| +C\max_{ i+j+2k =r, k\ge 1} (\da(x_0, y_0))^{j/2} |{\mathcal{D}}^{(i)} \p_2^{j+k} f (x_0, y_0)|\\
		& + C \max_{2(i+j)=r} |{\mathcal{D}}^{(i)} \p_2^j f(x_0, y_0)|.
	\end{align*} 
	It then follows by the  induction hypothesis that 
	\begin{align*}
		|S_1(A_v, x_0, y_0)|\leq |F^{(r)} (t_0)| + C n^r\|f\|_{L^\infty(G(\ld))}.
	\end{align*}
	On the other hand, since the function $F(t):=f(\Phi_{A_v}(z_0,t))$ 
	is an algebraic polynomial of degree $\le 2n$,   by  the univariate Bernstein  inequality, we obtain
	\[
	|F^{(r)}(t_0)|\le \|F^{(r)}\|_{C([-a_1,a_1])}\le c(a_1,a_0) n^r \|F\|_{C([-a_0,a_0])}\le c n^r \|f\|_{L^\infty(G(\ld))},
	\]
	where the last step uses Lemma~\ref{lem-1-1:Berns}  (i).
	Summarizing the above, we derive~\eqref{10-32-0} and hence complete the proof of  
	\eqref{10-28} for $p=\infty$.

	Next, we  prove~\eqref{10-28} for  $0< p<\infty$. By Lemma~\ref{lem:A_0 A_r}, it is enough to show that for each  $0\le v\le r$,
	\begin{equation}\label{eqn:s1_est}
		\iint_G |S_1 (A_v,x,y)|^p\, dxdy \le c n^{rp} \|f\|^p_{L^p(G(\ld))}.
	\end{equation}
	Fix $0\leq v\leq r$ and set 
	\[
	I_\ell :=\iint_G |S_\ell (A_v,x,y)|^p\, dxdy, \quad \ell=1,2,3.
	\]	
	Using~\eqref{s2},~\eqref{s3} and the induction hypothesis, we have 
	\begin{align}
		I_2& \leq C \max_{\sub{i+j+2k =r\\
				\  k\ge 1}} \iint_G \Bl|(\da(x,y))^{j/2} |{\mathcal{D}}^{(i)} \p_2^{j+k} f(x,y)\br|^p\, dxdy \leq C n^{rp} \|f\|_{L^p(G(\ld))}^p,\label{i2}\\
		I_3&\leq C \max_{2(i+j)=r} \iint_G |{\mathcal{D}}^{(i)} \p_2^j f(x,y)|^p\, dxdy \leq C  n^{rp} \|f\|_{L^p(G(\ld))}^p.\label{i3}
	\end{align}
	On the other hand,  performing the change of variables $x=z+t$ and $y=Q_{A_v}(z,t)$ and using Lemma~\ref{lem-1-1:Berns} (i), (iii),
	we obtain 
	\[
	I_\ell:= \iint_{(\Phi_{A_v}^+)^{-1} (G)} |S_\ell(A_v,z+t, Q_{A_v}(z,t))|^p (A_v +g''(z)) t \, dzdt, \quad \ell=1,2,3.
	\]
	Thus, setting 
	\[
	I: = \iint_{(\Phi_{A_v}^+)^{-1} (G)} \Bl| \f {d^r} {dt^r} \Bl[ f(z+t, Q_{A_v}(z,t))\Br]\Br|^p (A_v +g''(z)) t\, dzdt,
	\]
	and using~\eqref{i2},~\eqref{i3} and Lemma~\ref{lem:bern repres},
	we obtain 
	$$I_1\le c_p(I+I_2+I_3)\leq c_pI + c_p n^{rp} \|f\|_{L^p(G(\ld))}^p.$$
	
	Thus, for the proof~\eqref{eqn:s1_est},  we reduce to  showing  that \begin{align}
		I: &= \iint_{(\Phi_{A_v}^+)^{-1} (G)} \Bl| \f {d^r} {dt^r} \Bl[ f(z+t, Q_{A_v}(z,t))\Br]\Br|^p (A_v +g''(z)) t\, dzdt\notag\\
		&\leq C   n^{rp} \|f\|_{L^p(G(\ld))}^p.\label{10-37}
	\end{align}
	The proof is very similar to that of~\eqref{10-23-0}   for $r=1$, which is  given in the last subsection.
	Indeed,  using~\eqref{10-24-0}, Lemma~\ref{lem-1-1:Berns} and the   univariate weighted  Bernstein inequality, we have 
	\begin{align*}
		I	&\leq c \int_{-a-a_1 }^{a} \Bl[ \int_{0}^{a_1}\Bl| \f {d^r} {dt^r} \Bl[ f(z+t, Q_{A_v}(z,t))\Br] \Br|^p |t|\, dt\Br]\, dz\\
		& \leq c n^{rp} \int_{-a-a_1 }^{a} \Bl[ \int_{-a_0}^{a_0}\Bl| f(\Phi_{A_v}(z,t)) \Br|^p |t|\, dt\Br]\, dz	\leq c  n^{rp} \iint_{E_{A_v}} |f(\Phi_{A_v}(z,t))|^p|t|\, dzdt\\
		&\leq c n^{rp} \Bl[ \iint_{\Phi(E_{A_v}^{+})}  |f(x,y)|^p dxdy+\iint_{\Phi(E_{A_v}^{-})}  |f(x,y)|^p dxdy \Br]\leq c n^{rp} \|f\|_{L^p(G(\ld))}^p.	 
	\end{align*}
	This proves~\eqref{10-37} and hence completes the proof of the theorem.
\end{proof}

\section{Bernstein type inequality on domains of special type  in $\RR^d$}
\label{sec:13}
In this section, we shall  extend the Bernstein type inequality,  stated in  Theorem~\ref{THM:2D BERN}, 
 to higher-dimensional  domains of special type.  One of  the main difficulties in this case comes from the fact that we have  to deal with non-commutative   mixed directional derivatives along different  tangential directions. 
For convenience, we often    write a point in $\RR^d$ in the form  $(x, y)$ with $x\in\RR^{d-1}$ and $y=x_d\in \RR$. Sometimes  we  also  use a Greek letter to denote a point in $\RR^d$ and write it in the form $\xi=(\xi_x, \xi_y)$ with $\xi_x\in\RR^{d-1}$ and $\xi_y\in\RR$.

Let $d\ge 3$. Let   $G\subset \RR^d$  be  an $x_d$-upward domain with base size $a>0$ and parameter $L\ge 1$  given by 
\begin{equation}\label{11-2-2} G:=\Bl\{ (x, y)\in\RR^d:\  \ x\in (-a,a)^{d-1},\   \  g(x)-La< y\leq g(x)\Br\},\end{equation}
where $g:\RR^{d-1}\to \RR$ is a $C^2$-function  satisfying  that $\min_{x\in [-2a, 2a]^{d-1}}g(x)= 4La$.
Following the notations in Section \ref{decom-lem}, we have
\begin{align*}
G^\ast:&=\Bl\{( x,   y):  x\in (-2a,2a)^{d-1},   0 <  y \leq g(x)\Br\},
\end{align*}
and  for each  $\mu>0$, 
\begin{align*}
G(\mu):&=\{ (x, y):\  \ x\in (-\mu a,\mu a)^{d-1},\   \  g(x)-\mu L a<y\leq g(x)\},\\
\p' G (\mu) :&=\{ (x, g(x)):\  \   x\in (-\mu a, \mu a)^{d-1}\},\    \   \ \p' G=\p' G(1),\  \ \p'G^\ast=\p'G(2).
\end{align*}  
For $(x,y)\in G^\ast$,  we define  
\begin{equation*}
\da(x,y):=g(x)-y\   \   \text{and}\  \   \
\vi_n(x,y) :=\sqrt{\da(x,y)} +\f 1n,\   \  n=1,2,\dots.
\end{equation*}
According to   Lemma~\ref{lem-9-1}, we have
\begin{equation*}
\da(x,y)=g(x)-y \sim \dist (\xi, \p' G^\ast),\   \    \  \forall\xi= (x,y)\in G.
\end{equation*}


The  Bernstein type inequality on the domain  $G$ is formulated in terms of  certain  tangential derivatives along the essential boundary $\p' G$ of $G$, whose definition is given as follows.
For  $x_0\in [-2a, 2a]^{d-1}$,  let \begin{equation}\label{11-4-0}
\xi_j (x_0) :=e_j + \p_j g(x_0) e_{d},\   \ j=1,\dots, d-1
\end{equation}
be the tangent vector to $\p' G$ at the point  $(x_0, g(x_0))$ that is parallel to   the $x_jx_d$-coordinate plane. We  denote by 	$\p_{\xi_j(x_0)}^\ell$ the $\ell$-th order  directional derivative along  the   direction  of $\xi_j(x_0)$: 
\begin{equation*}\label{11-3}
	\p_{\xi_j(x_0)}^\ell:=(\xi_j(x_0)\cdot\nabla )^\ell=\sum_{i=0}^\ell \binom{\ell}i (\p_jg(x_0))^i \p_j^{\ell-i} \p_d^i,\   \ 
\end{equation*}
where $j=1,2,\dots, d-1$ and $ x_0\in [-2a,2a]^{d-1}.$ Thus, 
for   $(x,y)\in G$ and   $  f\in C^1(G)$,
$$  \p_{\xi_j(x)}^\ell f(x,y)=\sum_{i=0}^\ell \binom{\ell}i (\p_jg(x))^i( \p_j^{\ell-i} \p_d^i f)(x,y),\  \ 1\leq j<d. $$

  We also need to deal with  certain mixed directional derivatives.   Let $\NN_0$ denote the set of all nonnegative integers.  For $\pmb{\al}=(\al_1,\dots, \al_{d-1})\in \NN_0^{d-1}$, we set $|\pmb\al| =\al_1+\al_2+\dots+\al_{d-1}$, and define 
\begin{align*}
	{\mathcal{D}}_{\tan, x_0}^{\pmb{\al}}  =\p_{\xi_1(x_0)}^{\al_1} \p_{\xi_2( x_0)}^{\al_2} \dots \p_{\xi_{d-1}( x_0)}^{\al_{d-1}},\  \    \    \  x_0\in [-2a, 2a]^{d-1}.\label{11-4}
\end{align*}
For $(x,y)\in G$ and   $  f\in C^1(G)$, we also define
$${\mathcal{D}}_{\tan} ^{(\pmb{\al})}  f(x,y):=({\mathcal{D}}_{\tan, x}^{\pmb{\al}}  f)(x,y). $$
Note that  each of  the tangential  operators ${\mathcal{D}}_{\tan}^{(\pmb{\al})}$ defined above is commutative with $\p_{d}$, but the operators ${\mathcal{D}}_{\tan}^{(\pmb{\al})}$, $\pmb\al\in\NN_0^{d-1}$  themselves are not commutative. 
Finally, let $M>10$  be  a constant satisfying  that 
\begin{equation}\label{11-5-2}
	\max_{1\leq i\leq j\leq d-1} \|\p_i\p_j g\|_{L^\infty([-2 a,2a]^{d-1})}\le M\  \  \text{and}\  \  \|\nabla g (0)\|\le M.
\end{equation}

We will keep the above notations and assumptions  throughout  this section.

Our aim in this section is to prove the following theorem, which is a  higher dimensional extension of Theorem~\ref{THM:2D BERN}. 

\begin{thm}\label{thm:2d bern-2} Let   $G\subset \RR^d$  be  the  domain  given  in \eqref{11-2-2}  , and let $M>10$ be the constant satisfying \eqref{11-5-2}.
	If   $0<p\leq \infty$, $\ld>1$ and      $f\in \Pi_n^{d}$, then  for any $\pmb{\al}\in\NN_0^{d-1}$, 
	\begin{equation*}\label{11-6-18a}
		\|\vi_n^{i}{\mathcal{D}}_{\tan}^{(\pmb\al)}\partial_{d}^{i+j} f\|_{L^p(G)}\le  c n^{|\pmb \al|+i+2j}\|f\|_{L^p(G(\ld))},\  \  i,j=0,1,\dots, 
	\end{equation*}
	where  $c$ is a positive constant depending  only on $M$, $d$, $\ld$, $\pmb{\al}, i,j$,  and $p$.
\end{thm}

As in the case of $d=2$, Theorem \ref{thm:2d bern-2} also  allows us to deduce a slightly stronger Bernstein type inequality: 

\begin{cor}\label{cor-11-2} Let $\ld >1$ and $\mu>1$ be two given parameters.  If  $0<p\leq \infty$  and       $f\in\Pi_n^{d}$,  then for any $\pmb{\al}\in\NN_0^{d-1}$, and  $ i,j=0,1,\dots$,
	\begin{align*}
		\Bl\|\vi_n(\xi)^{i} &\max_{ u\in \Xi_{n,\mu,\ld}(\xi)}\Bigl|  {\mathcal{D}}_{\tan, u}^{\pmb{\al}}\partial_{d}^{i+j}f(\xi)\Br|\Br\|_{L^p(G; d\xi)} \le  c_\mu n^{|\pmb{\al}|+2j+i}\|f\|_{L^p(G(\ld))},	\end{align*}
	where 
	$$\Xi_{n, \mu,\ld} (\xi):= \Bl\{ u\in [-\ld a, \ld a]^{d-1}:\  \  \|u-\xi_x\|\leq \mu \vi_n(\xi)\Br\}.$$
\end{cor}

\begin{rem}  One can also establish similar results  on a  more general domain $G\subset \RR^d$ of special type as defined in Definition~\ref{def:specialtype}, which  can be  deduced directly  from Theorem \ref{thm:2d bern-2} through  translations and reflections of the domain.

\end{rem}

The proof of Corollary~\ref{cor-11-2} is almost identical to that of Corollary~\ref{cor-10-2} for $d=2$, so we skip the details here.
The rest of this section is devoted to the proof of Theorem~\ref{thm:2d bern-2}.

In the case  when $\mathcal{D}_{\tan,x}^{\pmb \al} =\p_{\xi_1(x)}^r$ for some positive integer $r$,  Theorem~\ref{thm:2d bern-2} can be deduced directly from Theorem~\ref{THM:2D BERN} and Fubini's theorem, as the next lemma shows.

\begin{lem}\label{lem-11} Let $\xi_1(x)$ be the vector given in \eqref{11-4-0}. 
	If   $0<p\leq \infty$, $\ld>1$ and      $f\in \Pi_n^{d}$, then  
	\begin{equation}\label{11-6-18b}
		\|\p_{\xi_1(x)}^{r} f(x,y)\|_{L^p(G)}\le  c n^{r}\|f\|_{L^p(G(\ld))},\  \  r=0,1,\dots, 
	\end{equation}
	where  $c$ is a positive constant depending  only on $M$, $d$, $\ld$, $r$,  and $p$, and the $L^p$-norm on the left hand side is computed with respect to the variables $x$ and $y$ and the Lebesgue measure on $G$. 
\end{lem}

\begin{proof}  For simplicity, we assume that $d=3$ and $p<\infty$. (The proof below works equally well for $d>3$ or $p=\infty$.)
		First, by Fubini's theorem, we have
	\begin{align}\label{2-2-0:bern}
		\iiint\limits_G &  \Bl|\p_{\xi_1(x_1, x_2)} ^r f(x_1, x_2, y)\Br|^pdx_1dx_2dy =\int_{-a}^a  I(x_2) \, dx_2,	
	\end{align}
	where 
	$$ I(x_2) :=  \int_{-a}^a  \int_{g(x_1, x_2)-La}^{g(x_1, x_2)} \Bl|\bigl(\p_1 +\p_1 g(x_1, x_2) \p_3\bigr)^r f(x_1, x_2, y)\Br|^p\, dy \, dx_1.$$
	For each fixed $x_2\in [-a, a]$, applying Theorem~\ref{THM:2D BERN}  to the function $g(\cdot, x_2)$ and the polynomial $f(\cdot, x_2, \cdot)$, we obtain
	\begin{align*}
		I(x_2) & \leq C n^{rp}  \int_{-\ld a}^{\ld a} \int_{g(x_1, x_2) -\ld L a}^{g(x_1, x_2)} |f (x_1,  x_2, y)|^p\, dy dx_1.
	\end{align*}
	Integrating this last inequality over $x_2\in [-a,a]$,  we deduce  \eqref{11-6-18b} from \eqref{2-2-0:bern}.
\end{proof}

 For  the proof of Theorem~\ref{thm:2d bern-2} in the general case, we need to extend  Lemma~\ref{lem-11}   to the case of more general  directional derivatives, defined below.  Given   $\xi\in\SS^{d-2}$ and $f\in C^\ell (G)$, define 
 \begin{align*}
\wt {D}^{(\ell)} _{\xi} f(x_0,y_0) &=\p_{(\xi, \p_\xi g(x_0))}^\ell f(x_0, y_0) = \bigl(\bigl(\xi, \p_\xi g(x_0)\bigr)\cdot \nabla \bigr)^\ell f(x_0, y_0)\\
&=\sum_{i=0}^\ell \binom{\ell} i (\p_\xi g(x_0))^i ( \p_\xi^{\ell-i}  \p _d^i f )(x_0, y_0),
 \end{align*}
 where $ (x_0, y_0)\in G$.
 Note here   that  $(\xi, \p_\xi g(x_0))$
 is the tangent vector   to $\p' G$ at the point $(x_0, g(x_0))$  that is   parallel to   the plane spanned by the vectors $(\xi,0)$ and $e_{d}$.
 By the definition, we have  
\begin{equation*}\label{directional}
	\wt {D}^{(\ell)} _{\xi} f(x,y)=\Bl(\f d{dt}\Br)^\ell \Bl[ f(x+t \xi, y+t \p_\xi g(x))\Br]\Bl|_{t=0},\   \  (x,y)\in G,\   \  f\in C^\ell (\RR^d),
\end{equation*}
and 
$$\wt D_{e_j} ^{(\ell)} f(x,y) =\p_{\xi_j(x)}^\ell f(x,y),\   \ 1\leq j\leq d-1.$$


As an extension of  Lemma~\ref{lem-11}, we have 

\begin{lem}\label{lem-11-2} 
	If   $0<p\leq \infty$, $\ld>1$ and      $f\in \Pi_n^{d}$, then  
	\begin{equation}\label{11-6-18c}
		\max_{\xi\in\SS^{d-2}} 	\| \wt{D}^{(r)}_\xi f\|_{L^p(G)}\le  c n^{r}\|f\|_{L^p(G(\ld))},\  \  r=0,1,\dots, 
	\end{equation}
	where  $c$ is a positive constant depending  only on $M$, $d$, $\ld$, $r$,  and $p$.
\end{lem}

\begin{proof}   
	Without loss of generality, we may assume that  $p<\infty$ since  the case $p=\infty$ is simpler and can be treated similarly.
		Let  $\xi\in\SS^{d-2}$ be a fixed direction. For $\mu\ge 1$,
		we write 
		$$[-\mu a, \mu a]^{d-1} =\Bl\{\eta+t\xi:\  \ \eta\in E_{\mu,\xi},\  \   t\in [a_{\mu,\xi} (\eta), b_{\mu,\xi}(\eta)]\Br\},$$
		where 
	$$  E_{\mu,\xi}: =\Bl\{ x-(x\cdot \xi) \xi:\  \  x\in [-\mu a, \mu a]^{d-1}\Br\},$$
	and 
	$$[a_{\mu,\xi}(\eta), b_{\mu,\xi}(\eta)]=\Bl\{t\in\RR:\   \  \eta+t\xi \in [-\mu a, \mu a]^{d-1}\Br\},\   \   \ \eta\in E_{\mu, \xi}.$$ 
 By Fubini's theorem, we then  have
	\begin{align*}
		\int_G &| \wt {D}^{(r)} _{\xi} f(x,y)|^p dxdy=  \int_{E_{1,\xi}}\Bl[ \int_{a_{1, \xi}(\eta)}^{b_{1, \xi}(\eta)}\int_{g_\eta(s)-La}^{g_\eta(s)} \Bl|  \f {d^r} {dt^r} \Bl[ f_\eta \bl(s+t,y+tg_\eta'(s)\br)\Br]\Br|_{t=0} \Br|^p \,  dy\, ds\Br] d\eta\\
		&= \int_{E_{1,\xi}} \int_{a_{1, \xi}(\eta)}^{b_{1, \xi}(\eta)}\int_{g_\eta(s) -La}^{g_{\eta}(s)} \Bl|  (\p_1 +g_\eta'(s)\p_2)^r f_\eta (s,y) \Br|^p \,  dy\, ds d\eta=:I,
	\end{align*}
where 
	\begin{align*}
		g_\eta (s) &:=g (\eta+s\xi),\    \ 
		f_{\eta}(s, y): =f( \eta+s\xi, y),\   \    y,s\in\RR,\   \ \eta\in\R^{d-1}.
	\end{align*} 
	
	Next, let  $1<\mu <\ld$ be a fixed parameter.  A straightforward calculation shows that   for each  $\eta\in E_{1, \xi}$, \begin{align*}
		&b_{\mu, \xi}(\eta)-a_{\mu, \xi}(\eta) \ge 2(\mu-1)a>0,\   \ \text{and}\\
		& 	\Bl [a_{\mu,\xi} (\eta)-(\ld-\mu)a, b_{\mu,\xi} (\eta) +(\ld-\mu) a\Br]\subset [a_{\ld,\xi} (\eta), b_{\ld,\xi} (\eta)].
	\end{align*}
	It then follows from Theorem~\ref{THM:2D BERN} that 
	\begin{align*}
		I &\leq  \int_{E_{1,\xi}} \int_{a_{\mu, \xi}(\eta)}^{b_{\mu, \xi}(\eta)}\int_{g_\eta(s) -La}^{g_{\eta}(s)} \Bl|  (\p_1 +g_\eta'(s)\p_2)^r f_\eta (s,y) \Br|^p \,  dy\, ds d\eta\\
		&\leq C n^{pr}  \int_{E_{1,\xi}}\Bl[ \int_{a_{\ld, \xi} (\eta)}^{b_{\ld, \xi}(\eta) }\int_{g_\eta(s)-\ld L a }^{g_\eta(s)} |  f_\eta(s,y)|^p \,  dy\, ds\Br] d\eta
		\leq C n^{r p} \int_{G(\ld)} |f(\varsigma)|^p\, d\varsigma.
	\end{align*} 
	This proves \eqref{11-6-18c}.	
	%
	%
\end{proof}

To deal with the  mixed directional derivatives $\mathcal{D}_{\tan}^{(\pmb \al)} $, one main difficulty lies in the fact that these  operators are not commutative.  To overcome the difficulty, 
we will use  a combinatorial identity   on mixed directional derivatives on $\RR^n$.   Recall that   $\p_\xi:=\xi\cdot\nabla=\sum_{j=1}^n \xi_j\p_j$ for each  $\xi=(\xi_1,\dots,\xi_n)\in\RR^n$.   As  a   consequence of the Kemperman lemma on mixed differences (see~\cite[(3.7)]{CD} or~\cite[Lemma 4.11, p.338]{BS}), we have 
\begin{lem}\label{lem-11-3}
	Let $\xi_1,\dots, \xi_r$ be arbitrary vectors in $\RR^n$.
	Then 
	\begin{equation}\label{11-12}
		\p_{\xi_1}\p_{\xi_2}\dots\p_{\xi_r}=\sum_{ S \subset \{1,2,\dots, r\}} (-1)^{\# S}  \p_{\xi_S}^r,
	\end{equation}
	where  
	the sum is taken  over all subsets $S$ of $\{1,2,\dots, r\}$, $\# S$ is the cardinality of $S$ and  $\xi_S =-\sum_{j\in  S} j^{-1} \xi_j$. 
\end{lem}
\begin{rem}
	In the case when all the vectors $\xi_1,\dots, \xi_r$ are unit, the following interesting inequality with sharp constant was proved in \cite{CD}:
	$$|\p_{\xi_1}\dots\p_{\xi_r} f(x)|\leq \max_{\xi\in \SS^{n-1}} |\p_\xi^r f(x)|,\   \  x\in\RR^n,\   \  \xi_1,\dots,\xi_r\in\SS^{n-1},\  \  \ f\in C^r(\RR^n). $$
	Note that in Lemma~\ref{lem-11-3}  the vectors $\xi_1,\dots, \xi_r$ are not necessarily unit, and the sum on the right hand side of~\eqref{11-12} uses at most $2^r$ directions.   
	This is very important in our later applications. 	
	
\end{rem}

\begin{proof}  For $x, h\in\RR^n$ and $f:\RR^n\to \RR$,  define 
	$$ T_h f(x)=f(x+h),\   \   \tr_h f(x)=(T_h-I)f(x)=f(x+h)-f(x).$$ 
	The proof relies on the following  identity   on finite differences, whose proof can be found in  \cite[Lemma 4.11, p.338]{BS}:
	\begin{equation}\label{kemperman} \tr_{h_1} \tr_{h_2}\dots \tr_{h_r}=\sum_{ S \subset \{1,2,\dots, r\}} (-1)^{\# S} T_{h_S^\ast} \tr_{h_S}^r,\   \   \  \forall h_1,\dots, h_r\in \RR^n, \end{equation}
	where   $$h_S^\ast = \sum_{j\in S} h_j\   \   \  \text{and}\   \   \  h_S =-\sum_{j\in S} j^{-1} h_j.$$
	For each   $f\in C^r(\RR^n)$ and  $x, \xi_1,\dots,\xi_r\in\RR^n$,  we have 
	\begin{align*}
		\tr_{t\xi_1} \tr_{t\xi_2} \dots  \tr_{t\xi_r} f(x)=\int_{[0,t]^r} (\p_{\xi_1}\dots\p_{\xi_r} f)\bl(x+ \sum_{j=1}^r u_j\xi_j\br)\,  du_r\dots du_1,  \  \  \forall t>0,
	\end{align*}
	which, in particular, implies 
	\begin{align*}
		\lim_{t\to 0+}  	t^{-r}	\bl(\tr_{t\xi_1} \dots \tr_{t\xi_r} f\br)(x)=\Bl(\p_{\xi_1}\dots \p_{\xi_r}f\Br)(x),\  \  f\in C^r(\RR^n).
	\end{align*}
	On the other hand, however,  using~\eqref{kemperman} with $h_j=t\xi_j$, we obtain
	\begin{align*}
		t^{-r}\Bl(	\prod_{j=1}^r \tr_{t\xi_j} \Br)f(x)= t^{-r}\sum_{ S\subset \{1,2,\dots, r\}} (-1)^{\# S} \tr_{t\xi_S}^r f\Bl(x+t\sum_{j\in S} \xi_j\Br).
	\end{align*}
	\eqref{11-12} then follows by letting $t\to 0+$. 
\end{proof}

Now we are in a position to prove Theorem~\ref{thm:2d bern-2}:\\

\begin{proof}[Proof of Theorem~\ref{thm:2d bern-2}]Since the operators $\mathcal{D}_{\tan}^{(\pmb \al)} $ and $\p_{d}$ are commutative, by the univariate Bernstein inequality (or the higher-dimensional analogue of Lemma~~\ref{lem:univ BM}),  it is sufficient to prove that 
	\begin{equation*}\label{11-14}
		\|{\mathcal{D}}_{\tan}^{(\pmb\al)} f\|_{L^p(G)}\le  c n^{|\pmb \al|}\|f\|_{L^p(G(\ld))},\  \  \   \  f\in\Pi_n^{d}.
	\end{equation*}

	Let  $(x_0, y_0)$ be an arbitrarily fixed  point in $G$. 
	Let $\ell=|\pmb{\al}|$ and let $\{\eta_j(x_0)\}_{j=1}^\ell$ be the  sequence of vectors such that 
	$$\eta_{\al_{j-1}+1} (x_0)=\dots=\eta_{\al_{j-1} +\al_j}(x_0) =\xi_j(x_0),\   \ j=1,2,\dots, d-1,$$ where $\al_0=0$.
	Then using~\eqref{11-12}, we have that 
	\begin{align}\label{11-15-0}
		& \p_{\xi_1(x_0)}^{\al_1} \dots \p_{\xi_{d-1}(x_0)}^{\al_{d-1}} =\prod_{j=1}^\ell \p_{\eta_j(x_0)} 
		= \sum_{S\subset \{1,2,\dots, \ell\}} (-1)^{\# S} \p_{\eta_S(x_0)}^\ell, 
	\end{align}
	where 
	$\eta_S(x_0) = -\sum_{j\in S} j^{-1}\eta_j(x_0).$
	As can be easily seen from~\eqref{11-4-0},  for each $S\subset \{1,2,\dots, \ell\}$, the vector $\eta_S (x_0)$ can be written in the form
	$$ \eta_S (x_0)=c_S\Bl(\xi_S, \p_{\xi_S} g(x_0)\Br),\   \  c_S>0,\  \   \xi_S\in\SS^{d-2},$$
	where  $c_S$ and $\xi_S$ depend  only on the set $S$ and  $\pmb{\al}$ (but independent of $x_0$).   
	Thus, by~\eqref{11-15-0}, it follows that  
	\begin{align}
		\Bl| \p_{\xi_1(x_0)}^{\al_1} \dots \p_{\xi_{d-1}(x_0)}^{\al_{d-1}} f(x_0, y_0)\Br|\leq C_{\pmb\al } \sum _{S\subset \{1,2,\dots, \ell\}} |\wt{D}_{\xi_S}^{(\ell) } f(x_0, y_0)|. \label{11-16}
	\end{align}
	Since the maximum in~\eqref{11-16} is taken over a set of $2^\ell$ elements, and all the vectors $\xi_S$ are independent of $(x_0, y_0)$, 
	it follows by Lemma~\ref{lem-11-2}  that 
	\begin{align*}
		\|{\mathcal{D}}_{\tan}^{(\pmb\al)} f\|_{L^p(G)}&\le C \sum_{S\subset \{1,\dots, \ell\}} \|
		\wt{D}_{\xi_S}^{(|\pmb \al|)} f\|_{L^p(G)} \leq 
		c n^{|\pmb \al|}\|f\|_{L^p(G(\ld))}.
	\end{align*}
	This completes the proof. 
\end{proof}

\section{Bernstein type inequality on general $C^2$ domains}\label{sec:14}

 In this section, we will use Proposition  ~\ref{lem-2-1-18} to 
  extend  
the Bernstein type inequality,  established in the last section, to a  more general $C^2$-domain. 
 Our aim  is to prove Theorem \ref{thm-10-1-00}.  
 
We first recall  some necessary  notations. Let $B_r(\xi):=\{\eta\in\RR^d:\  \ \|\eta-\xi\|< r\}$
denote  the open ball with center $\xi\in\RR^d$ and radius $r>0$   in $\R^{d}$.  Given a ball $B=B_r(x)\subset\RR^d$, and a constant $c>0$,   we  denote by $cB=B_{cr}(x)$ the dilation of  $B$ from its center by a factor $c>0$. 
Let   $\Og\subset \RR^{d}$ be a compact $C^2$-domain  with   boundary $\Ga=\p \Og$. 
Given $\eta\in\Ga$, we denote by $\mathbf{n}_\eta$  the unit outer normal vector   to  $\Ga$ at  $\eta$, and $\mathcal{S}_\eta$  the set of all unit tangent vectors to $\Ga$ at $\eta$: $$\mathcal{S}_\eta:=\{ \pmb{\tau}\in\SS^{d-1}:\  \ \pmb\tau\cdot \mathbf{n}_\eta=0\}.$$
For a parameter $\mu\ge 1$, set
	\begin{align*}
\Ga_{n, \mu} (\xi): &=\{ \eta\in \Ga:\  \ \|\eta-\xi\|\leq \mu \vi_{n,\Ga} (\xi)\},
\end{align*}
where 
$$\vi_{n,\Ga} (\xi):=\sqrt{ \dist(\xi, \Ga)} +n^{-1},\   \ \xi\in\Og,\  \ n=1,2,\dots.$$
Then for any two  nonnegative integers $l_1, l_2$,
$$ \mathcal{D}_{n,\mu}^{l_1, l_2} f(\xi):=\max_{\eta\in \Ga_{n,\mu} (\xi)}  \max_{\pmb{\tau}\in\mathcal{S}_\eta}
\bl|  \p_{\pmb{\tau}}  ^{l_1} \p_{\nb_\eta}^{l_2} f(\xi)\br|,\  \  \xi\in\Og,\   \  f\in C^\infty(\Og).$$
Our  goal is to show that for any     $f\in \Pi_n^{d}$,   $\mu>1$ and $0<p\leq \infty$, 
	\begin{equation}\label{bernstein-tan}
		\Bl\|\vi_{n, \Ga}^j \mathcal{D}_{n,\mu}^{r, j+l} f \Br\|_{L^p(\Og)}\leq C_\mu n^{r+j+2l} \|f\|_{L^p(\Og)},\  \ r,j,l=0,1,\dots, 
	\end{equation}  
	where the constant $C_\mu$ is independent of $f$ and $n$.

To show \eqref{bernstein-tan}, we first use 
 Proposition ~\ref{lem-2-1-18} and  cover $\Ga$ by finitely many domains $G_1,\dots, G_{m_0}\subset \Og$ of special type attached to $\Ga$, where $m_0$ depends only on the domain $\Og$. Since each $G_j$ is open with respect to  the relative topology of $\Og$,    there exists a constant  $\da\in(0,1)$ depending only on $\Og$  such that 
$$
\Ga_\da:=\Bl\{\xi\in\Og:  \dist(\xi, \Ga) \leq  \da\Br\} \subset \bigcup_{j=1}^{m_0} G_j.$$
Here  we may choose $\da\in (0,1)$ small enough so that the base size of each $G_j$ is bigger than $2\mu\sqrt{\da}$.

Second, we  cover $\Og_\da:=\Og\setminus \Ga_\da$ by finitely many open balls $B_1, \dots, B_{m_1}$ of radius $\f \da 2$  such that $m_1\leq C \da^{-d}$ and $\overline {2B_i}\subset \Og$ for $1\leq i\leq m_1$. 
Then by the univariate Bernstein inequality,  for each $1\leq i\leq m_1$, we have
\begin{equation*}\label{10-1}
\| \p^{\pmb{\b}} f \|_{L^p(B_i)}\leq C_{{\pmb \b}, \da}  n^{|\pmb{\b}|} \|f\|_{L^p(2B_i)},\   \  \forall \pmb {\b}\in \NN_0^{d}.
\end{equation*}
This implies  \begin{equation*}\label{10-1}
\| \p^{\pmb{\b}} f \|_{L^p( \Og_\da)}\leq C_{{\pmb \b}, \da}  n^{|\pmb{\b}|} \|f\|_{L^p(\Og)},\   \  \forall \pmb {\b}\in \NN_0^{d},
\end{equation*}
and hence
	\begin{equation}\label{Bernstein-tan-1}
		\Bl\|\mathcal{D}_{n,\mu}^{r, j+l} f \Br\|_{L^p(\Og_\da)}\leq C_\da n^{r+l+j} \|f\|_{L^p(\Og)},\  \  r, j,l=0,1,\dots. 
	\end{equation}

Finally, we prove that     for each domain $G\subset \Og$ of special type  attached to $\Ga$ with base size $a\ge 2\mu\sqrt \da$, 
\begin{equation}\label{bernstein-tan-2}
\Bl\|\vi_{n, \Ga}^j \mathcal{D}_{n,\mu}^{r, j+l} f\Br\|_{L^p(G\cap\Ga_\da)}\leq C n^{r+j+2l} \|f\|_{L^p(G^\ast)},\  \  r,j,l=0,1,\dots, 
\end{equation} 
which together with \eqref{Bernstein-tan-1} will imply the desired inequality \eqref{bernstein-tan}. Here we recall that the set $G^\ast$ is defined in \eqref{G}.

To show  \eqref{bernstein-tan-2}, without loss of generality, we may assume that $n>\f  {2\mu}a$, and 
 $G$ is an $x_{d}$-domain that takes    the  form~\eqref{11-2-2}, where $a\ge  2\mu\sqrt \da$, $L\ge 1$ and   $g\in C^2(\RR^{d-1})$ satisfies   $$\min _{x\in [-2a, 2a]^{d-1}} g(x)= 4L a.$$ 
				 Since $G\subset \Og$ is attached with $\Ga$, we have $G_\ast \subset \Og$ and $\p' G_\ast\subset \Ga$.  Moreover,  according to Proposition~\ref{metric-lem},  we have 
	\begin{equation}\label{12-5}
		\vi_n(x,y):=\sqrt{g(x)-y}+\f1n \sim \vi_{n,\Ga}(x,y),\   \    \  \forall (x,y)\in G. 
	\end{equation}
Since $a\ge 2\mu \sqrt{\da}$ and $n>\f  {2\mu}a$, it follows that 
$\mu \vi_{n,\Ga}(\xi) <a$ and $\Ga_{n,\mu}(\xi) \subset \p ' G_\ast$ for every  $ \xi\in G\cap \Ga_\da.$
 Thus, by  \eqref{12-5}, 
  there exists a constant $c_1>1$ such that  
 $$ \Ga_{n,\mu} (\xi) \subset \Bl\{ (u, g(u)):\  \   u\in \Xi_{n, c_1\mu} (\xi)\Br\}, \  \ \forall \xi\in G\cap \Ga_\da,$$ 
 where 
		$\Xi_{n, \mu} (\xi):= \Bl\{ u\in [-2a, 2a]^{d-1}:\   \   \|u-\xi_x\|\leq\mu \vi_n(\xi)\Br\}.$
	Thus, using Corollary~\ref{cor-11-2}, we obtain 
\begin{align}
&\Bl\| \vi_n(\xi)^j \max_{\eta \in\Ga_{n,\mu} (\xi)} 
\bl|  \mathcal{D} _{\tan, \eta_x} ^{\pmb \al} \p_{d}^{j+l} f(\xi) \br| \Br\|_{L^p(G\cap \Ga_\da; d\xi)} \leq   C n^{|\pmb\al|+j+2l} \|f\|_{L^p(G^\ast)},\label{6-7-a}
\end{align}
where $\eta=(\eta_x,\eta_y)$,  $\eta_x\in\RR^{d-1}$ and $\eta_y\in\RR$.

	Now let $\eta =(x_0, g(x_0))\in \p'G^\ast$ with  $x_0=\eta_x\in (-2a, 2a)^{d-1}$. Let  $\xi_j(x_0)$, $j=1,\dots, d-1$ be the vectors given  in~\eqref{11-4-0}.
	A straightforward calculation then shows that  
	\begin{align*}
	\p_{\pmb\tau}& =\sum_{j=1}^{d-1} \tau_j \p_{\xi_j(x_0)},\   \ \forall 	\pmb \tau=(\tau_1, \dots,\tau_d) \in \mathcal S_\eta\\
		\p_{\mathbf n_\eta} &=-\sum_{j=1}^{d-1} \f {\p_j g(x_0)} {\sqrt{ 1+\|\nabla g(x_0)\|^2}}\p_{\xi_j(x_0)} 	+\sqrt{ 1+\|\nabla g(x_0)\|^2} \p_d.
	\end{align*}
Recalling that  
	$$ \mathcal{D}_{\tan, \eta_x}^{\pmb{\al}} f(x, y) =  \p_{\xi_1(x_0)}^{\al_1} \dots \p_{\xi_{d-1}(x_0)}^{\al_{d-1}} f(x, y),\   \ {\pmb\al} \in\NN_0^{d-1}, $$
 we obtain 
  	\begin{align}
 \max_{\pmb{\tau} \in \mathcal{S}_\eta} 
 \bl|  \p_{\pmb{\tau}}  ^{r} \p_{\nb_\eta}^{j+l} f(\xi) \br|\leq C   \max_{\sub{|\pmb{\al}|+i =r+j+l\\
 		0\leq 	i\leq j+l}} 
 \bl|  \mathcal{D} _{\tan, \eta_x} ^{\pmb \al} \p_{d}^{i} f(\xi) \br|,\   \    \ \forall \xi\in G.\label{6-8b}
 \end{align}
  Taking maximum  over $\eta=(\eta_x,\eta_y) \in\Ga_{n,\mu}(\xi)$ on both sides of \eqref{6-8b} yields  that  for each $\xi\in G\cap\Ga_\da $,

	\begin{align}
		\mathcal{D}_{n,\mu} ^{ r, j+l} f(\xi) &:=\max_{\eta\in\Ga_{n,\mu} (\xi)} \max_{\pmb{\tau} \in \mathcal{S}_\eta} 
		\bl|  \p_{\pmb{\tau}}  ^{r} \p_{\nb_\eta}^{j+l} f(\xi) \br|\notag\\
		&\leq C \max_{\eta\in\Ga_{n,\mu} (\xi)}  \max_{\sub{|\pmb{\al}|+i =r+j+l\\
				0\leq 	i\leq j+l}} 
		\bl|  \mathcal{D} _{\tan, \eta_x} ^{\pmb \al} \p_{d}^{i} f(\xi) \br|. \label{6-8a}
	\end{align}

Now  combining \eqref{12-5}, \eqref{6-7-a} with \eqref{6-8a},  we obtain
\begin{align*}
\Bl\|\vi_{n, \Ga}^j \mathcal{D}_{n,\mu}^{r, j+l} f\Br\|_{L^p(G\cap\Ga_\da)}&\leq C  \max_{\sub{|\pmb{\al}|+i =r+j+l\\
		0\leq 	i\leq j+l}}   \Bl\| \vi_n (\xi) ^j \max_{\eta\in \Ga_{n,\mu} (\xi) } 	\bl|  \mathcal{D} _{\tan, \eta_x} ^{\pmb \al} \p_{d}^{i} f(\xi)\br| \Br\|_{L^p(G\cap\Ga_\da; d\xi)}\\
	&\leq  C\Bl[ \max _{\sub { |\pmb{\al}|+i =r+j+l\\
			j\leq i\leq j+l} }  n^{j +2(i-j) +|\pmb \al|}+ \max_{\sub{ |\pmb{\al}|+i =r+j+l\\
			 0\leq i<j}}  n^{i +|\pmb \al|}    \Br ]\|f\|_{L^p(G^\ast)}\\
&\leq C n^{r+j+2l} \|f\|_{L^p(G^\ast)},\  \  r,j,l=0,1,\dots.
\end{align*}
This proves the  desired inequality ~\eqref{bernstein-tan-2}, and hence   completes the proof of the theorem.


\section
{Marcinkiewicz type  inequality
	and positive cubature formulas}\label{sec:16}

%
%
%


In this section, we will prove  Theorem \ref{cor-16-2-0} and Theorem \ref{cor-16-3-0}. We start with a brief description of  some necessary notations. Let $(X,\rho)$ be a metric space. Given $\da>0$, a  finite subset    $\Ld$ of $X$  is said to be    $\da$-separated  with respect to the metric $\rho$    if $\rho(\og, \og') \ge \da$ for any two distinct points $\og, \og'\in\Ld$, while a  $\da$-separated subset $\Ld\subset X$ is called maximal   if
$\inf_{\og\in\Ld} \rho(x, \og)<\va$ for any $x\in X$. 
As usual, we assume that $\Og\subset \RR^{d}$ is  a compact $C^2$-domain with boundary $\Ga=\p \Og$. Let $\rho\equiv \rho_{\Og}$ denote the metric on $\Og$ defined by~\eqref{metric}.  Let   $ U( \xi, t):= \{\eta\in\Og:\  \  \rho_\Og(\xi,\eta) \leq t\}$ for  $\xi\in\Og$ and $t>0$.  According to Corollary~\ref{rem-6-2} (i),  
\begin{equation*}\label{16-1-00}
	|U(\xi, t)| \sim t^{d} ( t +\sqrt{\dist (\xi, \Ga)}),\  \  \xi\in\Og,\   \  t\in (0,1),
\end{equation*}
while by Corollary~\ref{rem-6-2} (iii), for each $\da\in (0,1)$, 
every $\da$-separated subset $\Ld$ of $\Og$   must satisfy 
$\# \Ld \leq C \da^{-d}$.
Given a    bounded function $f$ on $\Og$ and a subset $I\subset \Og$,   we  define
$$\osc(f; I) :=\sup_{\xi,\eta\in I} |f(\xi)-f(\eta)|.$$

Our goal in this section is to prove the following result, from which  Theorem \ref{cor-16-2-0} and Theorem \ref{cor-16-3-0} will follow. 

\begin{thm}\label{thm-16-1:MZ} Let  $\Og\subset \RR^{d}$  be   a compact $C^2$-domain with boundary $\Ga=\p \Og$.
	Let $\ell\ge 1$ be a given parameter,  $n$  a positive integer, and let  $\va\in (0, \ell )$.  Assume that    $\Ld\subset\Og$ is  $\f \va n$-separated with respect to the metric $\rho_\Og$.  Then for any  $1\leq p\leq \infty$ and  $f\in\Pi_{n}^{d}$,
		\begin{equation}\label{3-4:bern}\left( \sum_{\xi\in\Ld} \left|U(\xi, \f \va n )\right|  \left( \osc \left(f; U \Bl(\xi, \f{\ell\va}n \Br)\right) \right)^p \right)^{\f 1p} \leq C_\ell\, \va  \|f\|_{L^p(\Og)},\end{equation}
	where  the constant $C_\ell>0$ depends only on the parameter $\ell$ and the domain $\Og$, and we have to replace  the term on the left hand side of \eqref{3-4:bern}  with $\max_{\xi\in \Ld}  \osc \bl(f; U (\xi, \f {\ell \va}n )\br) $ if  $p=\infty$.
\end{thm}

The rest of this section is organized as follows.  We first take Theorem \ref{thm-16-1:MZ} for granted and show how it implies  Theorem \ref{cor-16-2-0} and Theorem \ref{cor-16-3-0}  in  Section \ref{subsec:7-1} and Section \ref{subsec:7-2}, respectively.   Section \ref{subsec:7-2} is devoted to the proof of  Theorem \ref{thm-16-1:MZ}, which is more involved. The Bernstein inequality stated in Theorem~\ref{THM:2D BERN} will play a crucial role in our proof. 

\subsection{Proof of Theorem  \ref{cor-16-2-0}  (assuming Theorem \ref{thm-16-1:MZ})}\label{subsec:7-1}
Without loss of generality, we may assume that $p<\infty$. (The case $p=\infty$ can be deduced by letting $p\to\infty$ since all the implicit constants below are independent of $p$.)

Let $\{R_1,\dots, R_N\}$ be a  partition of $\Og$ with norm $\leq \f{ \da} n$. Let $\Ld$ be a maximal $\f \da n$-separated subset of $\Og$. For each $\og\in \Ld$, set
$$I_{\og}:=\Bl\{j\in\{1,2,\dots,N\}:\  R_j\cap U\Bl(\og, \f \da n\Br) \neq \emptyset\Br\}.$$
Then $\{1,2,\dots, N\} =\bigcup_{\og\in\Ld} I_\og$, and $\bigcup_{j\in I_\og} R_j\subset U(\og, \f {2\da} n)$ for any  $\og\in\Ld$.
%
Thus,  using Theorem \ref{thm-16-1:MZ}), we obtain that for any $f\in\Pi_n^d$,
\begin{align*}
&\sum_{j=1}^N (\osc (f; R_j) )^p |R_j| \leq  \sum_{\og\in\Ld} \Bl(\osc \bigl(f, U(\og, \tfrac{2\da} n)\bigr)\Br)^p \sum_{j\in I_\og}  |R_j|\\
& \leq  \sum_{\og\in \Ld} \Bl(\osc (f, U(\og, \tfrac{2\da} n))\Br)^p \Bl|U(\og, \tfrac {2\da} n)\Br| \leq C  \da^p \|f\|_p^p.
\end{align*}
%
%
%
Since $\xi_j\in R_j $ for each $1\leq j\leq N$, it follows by  the Minkowski inequality that 
\begin{align}
\Bl|\bigl( \sum_{j=1}^N |f(\xi_j)|^p  |R_j| \bigr)^{\f1p} -\|f\|_p \Br| \leq \Bl( \sum_{j=1}^N \bigl(\osc(f; R_j)\bigr)^p |R_j|\Br)^{\f1p} \leq  C\da \|f\|_p.\label{7-3a}
\end{align}
Thus, if   $0<\da\leq \da_0:=\f {1} {2C}$, then
\begin{align*}
\f 12 \|f\|_p \leq \bigl( \sum_{j=1}^N |f(\xi_j)|^p  |R_j| \bigr)^{\f1p}\leq \f 32 \|f\|_p,\   \  \forall f\in \Pi_n^d.
\end{align*}
\subsection{Proof of Theorem  \ref{cor-16-3-0} (assuming Theorem \ref{thm-16-1:MZ})}\label{subsec:7-2}
Let $\{R_1,\dots, R_N\}$ be a partition of $\Og$ with norm $\leq \f {\da_0} {n}$. Here and throughout the proof, without loss of generality, we may assume that $\da_0>0$ is a sufficiently small constant depending only on $\Og$. Some of the estimates below may not be true without this assumption.
It is easily seen from the proof of   \eqref{7-3a} that for any $f\in\Pi_n^d$, 
\begin{align*}  
&\Bl| \sum_{j=1}^N f(\xi_j) |R_j| -\int_{\Og} f(\xi)\, d\xi\Br| \leq  C{\da_0} \sum_{j=1}^N |f(\xi_j)| |R_j|,
\end{align*}
where $C>0$ is a constant depending only on $\Og$. 
Setting $\da_0:=\f 1 {3C}$, we deduce that for any    $f\in\Pi_n^d$ satisfying   $\min_{1\leq j\leq N} f(\xi_j) \ge 0$,  
\begin{align}
\f 23 \sum_{j=1}^N f(\xi_j) |R_j| \leq \int_{\Og} f(\xi)\, d\xi\leq \f 43 \sum_{j=1}^N f(\xi_j) |R_j|.\label{7-4-0}
\end{align}

Next, we  denote by $\delta_\xi$ the Dirac probability measure supported at $\xi\in\Og$, and  consider the following linear functional on $\Pi_n^d$:
\begin{equation}\label{7-5-0}
T f=\f 43 \f 1 {|\Og|} \int_{\Og} f(\xi)\, d\xi -\f 1{3|\Og|} \sum_{j=1}^N f(\xi_j)|R_j|,\   \ f\in \Pi_n^d.
\end{equation}
 We claim that 
\begin{equation}\label{7-5-1}
T \in \conv \Bl\{ \delta_{\xi_1},\dots, \da_{\xi_N}\Br\}\subset (\Pi_n^d)^\ast,
\end{equation}
which will imply 
that  there exist constants 
$\ld_j \ge \f 14 |R_j|$, $j=1,\dots, N$ such that 
\begin{equation}\label{cubature}
\int_{\Og} f(\xi)\, d\xi =\sum_{j=1}^N \ld_j f(\xi_j),\   \ \forall f\in\Pi_n^d.
\end{equation}

Assume to the contrary that \eqref{7-5-1} were not true. Then by the convex separation theorem, there exists $P\in\Pi_n^d$ such that 
\begin{equation}\label{7-7c}
T(P)  < \min_{1\leq j\leq N} P(\xi_j).
\end{equation}
Without loss of generality, we may assume that $ \min_{1\leq j\leq N} P(\xi_j)=0$ since otherwise we may replace $P$ with $P- \min_{1\leq j\leq N} P(\xi_j)$. Then \eqref{7-7c} implies 
$ T(P) <0.$
However,  this is impossible since by \eqref{7-4-0} and \eqref{7-5-0}
\begin{align*}
|\Og| T (P) \ge \f 59 \sum_{j=1}^N P(\xi_j) |R_j| \ge 0.
\end{align*}
This proves the claim \eqref{7-5-1}, and hence \eqref{cubature}.

Finally, we prove that 
\begin{equation}\label{eqn:lambda-bound}
\ld_j \leq C \Bl(\f {1} n\Br)^d\Bl( \f {1} n +\sqrt{\dist(\xi_j,\p \Og)} \Br) ,\   \  j=1,2,\dots,N.
\end{equation}
Indeed, for each $1\leq j\leq N$,  it suffices to find  a nonnegative polynomial $P_j\in\Pi_n^d$ on $\Og$  such that $P_j(\xi_j)=1$ and 
$\int_{\Og} P_j(\xi) \, d\xi \leq C n^{-d}(n^{-1}+\sqrt{\dist(\xi_j,\p \Og)})$ as then~\eqref{eqn:lambda-bound} immediately follows  from~\eqref{cubature} with $f=P_j$. The required construction is a standard technique in estimates of Christoffel function. Fix $j$, $1\leq j\leq N$, and let $\mu\in\p\Og$ be such that $\|\xi_j-\mu\|=\dist(\xi_j,\p \Og)=:a$. If $\mathbf n_\mu$ denotes the unit outer normal vector to $\p \Og$ at $\mu$, then $\xi_j=\mu-a \mathbf n_\mu$ and the definition of $C^2$-domains implies that there exist radii $r_1,r_2>0$ depending only on $\Omega$ such that
$\Omega \subset B_{r_2}(\mu+r_1\mathbf n_\mu)\setminus B_{r_1}(\mu+r_1\mathbf n_\mu)=:D$. Without loss of generality, we can assume $\mu+r_1\mathbf n_\mu=0$ and $\mathbf n_\mu=(1,0,\dots,0)$. A nonnegative polynomial $P_j$ on $D$  satisfying  
\[P_j(\xi_j)=1 \quad\text{and}\quad 
\int_{D} P_j(\xi) \, d\xi \leq C n^{-d}(n^{-1}+\sqrt{\dist(\xi_j,\p \Og)})
\] 
can be constructed as
\[
P_j(x_1,\dots,x_d)=(Q(\|\x\|^2)Q_2(x_2)\dots Q_d(x_d))^2,
\] 
where $Q$ and $Q_i$, $2\le i\le d$, are the polynomials provided by~\cite{Di-Pr16}*{Lemma~6.1} applied on the intervals $[r_1^2,r_2^2]$ and $[-r_2,r_2]$, respectively, with $y$ chosen to guarantee $P_j(\xi_j)=P_j(-a,0,\dots,0)=1$. We omit the details here referring an interested reader to the proof of~\cite{PrU2}*{Lemma~4.3} which contains the required construction and estimates for the case $d=2$.

This completes the proof of Theorem  \ref{cor-16-3-0}. 
          
%

\subsection{Proof of Theorem \ref{thm-16-1:MZ}}\label{subsec:7-3}
This subsection is devoted to the proof of Theorem~\ref{thm-16-1:MZ}.   We  will write the proof for the case of  $p<\infty$ only, as the case $p=\infty$ is simpler and can be treated similarly. 
 The proof relies on  several lemmas.

\begin{lem}\label{lem-3-1:bern} 
	Let  $0\leq \ta_1<\ta_2<\dots\leq\ta_m\leq  \pi$ satisfy that  $\min_{2\leq i\leq m} (\ta_i-\ta_{i-1})\ge \f 1n$ for some positive integer $n$.  
	Then for every  $f\in\Pi_k^1$,  $1\leq p<\infty$ and any parameter  $\ell>1$, 
	\begin{equation}\label{7.8}
	\Bl(\sum_{ i=1}^{m}\Bl(\f {\sin \ta_i} n  +\f1{n^2} \Br) \max_{ |\ta-\ta_i|\leq \f \ell n} |f(\cos\ta )|^p\Br)^{\f1p}\leq C_\ell  \Bl( 1+\f { k}n\Br) \Bl(\int_{-1}^1 |f(x)|^p\, dx\Br)^{\f1p}, 
	\end{equation}
	where $C_\ell>0$ is a constant depending only on the parameter $\ell$.
\end{lem}
\begin{proof} Lemma \ref{lem-3-1:bern}  can be proved by a slight modification of the proof of Theorem~3.1 of \cite{MT2}. For completeness, we summarize its proof as follows. 
	First, without loss of generality, we may assume  $k\ge n$, since  the stated inequality for $k\leq n$ follows directly from the case of  $k=n$.

Next, define $T_k(\ta):=f(\cos\ta)$ for $f\in\Pi_k^1$. 
 Let  $I_i:= [\ta_i-\f \ell n, \ta_i+\f \ell n]$ for $1\leq i\leq m$. 
 Note that  for each $\ta\in I_i$, 
 $$ |T_k(\ta)| \leq \int_{I_i} |T_k'(s)|\, ds +\f 1 {|I_i|} \int_{I_i} |T_k(s)|\, ds,$$
 which implies 
 \begin{align*}
\max_{\ta \in I_i} |T_k(\ta)|^p & \leq  C_\ell^p n^{-(p-1)} \int_{I_i} |T_k'(s)|^p\, ds+ C_\ell^p n \int_{I_i} |T_k(s)|^p\, ds.
\end{align*}
Since
\[  \sin \ta_i \leq \min_{\ta \in I_i} |\sin \ta | + C_\ell n^{-1},\  \ 1\leq i\leq m, \]
it follows that 
\begin{equation}
\Bl(\sum_{ i=1}^{m}\Bl(\f {\sin \ta_i} n  +\f1{n^2} \Br) \max_{ \ta\in I_i} |T_k(\ta)|^p\Br)^{1/p}\leq C_\ell (\Sigma_1+\Sigma_2),
\end{equation}
 where 
\begin{align*}
\Sigma_1:&= n^{-1} \Bl(\int_{-\pi}^\pi |T_k'(\ta)|^p \bigl(|\sin \ta| + n^{-1} \bigr)\, d\ta\Br)^{1/p},\\
\Sigma_2:&=  \Bl(\int_0^\pi  |T_k(\ta)|^p \bigl(\sin \ta   +n^{-1} \bigr)\, d\ta\Br)^{1/p}.
\end{align*}
For the term $\Sigma_1$,  we use   the  weighted Bernstein inequality for trigonometric polynomials (see \cite[Theorem 4.1]{MT2}) to obtain 
\begin{align*}
\Sigma_1&\leq  \f {Ck}n \Bl(\int_{0}^\pi |T_k(\ta)|^p (\sin \ta+n^{-1})\, d\ta\Br)^{1/p}= \f {Ck}n\Sigma_2.
\end{align*}

Finally, using the  Schur-type inequality for trigonemetric polynomials (see \cite[(3.3)]{MT2}), we obtain 
\begin{align*}
\Sigma_2\leq C \Bl(\int_0^\pi  |T_k(\ta)|^p \sin \ta  \, d\ta\Br)^{1/p}=C\Bl( \int_{-1}^1 |f(x)|^p dx\Br)^{1/p}.
\end{align*}

	Putting the above together, we deduce the desired inequality \eqref{7.8}. 
	\end{proof}

Our next lemma gives an analogue of \eqref{3-4:bern} on domains of special type.  Let $G\subset \Og$ be a domain of special type attached to $\Ga$. Without loss of generality, we may assume that $G$ is an $x_d$-upward domain with base size $b\in (0,1)$ and parameter $L=b^{-1}$, given by 
\begin{equation}\label{7.10}
 G:=\{(x,y):\  \  x\in (-b, b)^{d-1},\   \   g(x)-1<y\leq g(x)\},
\end{equation}
where  $g$ is a $C^2$-function on $\RR^{d-1}$ satisfying that $\min_{x\in [-2b,2 b]^{d-1}} g(x)= 4$.
Following the notation in  Section \ref{decom-lem}, we then have
$$ G^\ast:=\{(x,y):\  \  x\in (-2b, 2b)^{d-1},\   \   0<y\leq g(x)\}.$$
Let  $\wh{\rho}_G$ be the metric on $G$  defined by 
\begin{equation}\label{rhog-0}
	\wh{\rho}_G(\xi, \eta):=\max\Bl\{\|\xi_x-\eta_x\|,
	\Bl|\sqrt{g(\xi_x)-\xi_y}-\sqrt{g(\eta_x)-\eta_y}\Br|\Br\},
\end{equation}
where  $\xi=(\xi_x, \xi_y)\in G$ and  $\eta=(\eta_x, \eta_y)\in G$. 
According to Proposition~\ref{metric-lem}, we have 
\begin{equation*}\label{6-1-metric-0}\wh{\rho}_G(\xi,\eta)\sim \rho_{\Og} (\xi,\eta),\    \    \  \xi, \eta\in G.\end{equation*}
For $\xi\in G$ and $r\in (0,1)$, we  define  
$$ B_G (\xi, r):=\Bl\{ \eta\in G:\  \  \wh{\rho}_G(\xi,\eta)\leq r\Br\}.$$

The following lemma will play an important role in the proof of Theorem \ref{thm-16-1:MZ}:

\begin{lem}\label{lem-16-2} Let $G\subset \RR^d$ be the domain given in \eqref{7.10}.
	Assume that  $n\in\NN$,  $\mu\ge 1$ is  a parameter, and $\Ld\subset G$ is   $\da$-separated with respect to the metric $\wh\rho_G$ for some $\da\in (0,\f \mu n)$.  Then for any   $f\in\Pi_{n}^{d}$ and $1\leq p<\infty$, 
	\begin{equation}\label{3-4:berna} \Bl(\sum_{\og\in\Ld} | B_G (\xi, \da)|\Bl(\osc \bl(f; B_G \bl(\og, \mu\da \br )\br)\Br)^p\Br)^{\f1p} \leq C_\mu (n\da) \|f\|_{L^p(G^\ast)},\end{equation}
where  the constant $C_\mu$ is independent of $\da$, $n$ and $f$.

\end{lem}


\begin{proof}For simplicity,  we  assume  $d=2$. (The  proof below with slight modifications   works equally well for the case $d>2$.) Then 
	\begin{align*}
	G:&=\Bl\{(x,y):\  \  x\in (-b, b),\   \   g(x)-y\in [0,1)\Br\}\   \   \ \text{and}\\
	G^\ast:&=\{(x,y):\  \  x\in (-2b, 2b),\   \   0<y\leq g(x)\},
	\end{align*}
	where  $g$ is a $C^2$-function on $\RR$ satisfying that $\min_{x\in [-2b, 2b]} g(x)= 4$.

	

First, we construct a  partition of the domain $G$.
Let  $L:=\max_{x\in [-2b, 2b]}|g(x)|+10$, and let $n_1$ denote  the smallest   integer $>\f 1{c_0 \da}$, 
 where $c_0\in (0,1)$ is a  constant depending only on $\mu$ and $\Og$. Indeed, we will choose $c_0$   sufficiently small  so that $n_1>5 L+n$.
	Let
	$$\b_j:= L-L\cos \f {j\pi} {n_1}=2L \sin^2 \f {j\pi} {2n_1},\  \ j=0,1,\dots, n_1$$
	be the Chebyshev partition of the interval $[0, 2L]$ of order $n_1$.  Let $m$ denote the  positive   integer such that 
	$\b_m<1\leq \b_{m+1}$.
	We then partition the interval $[0,1]$ with  nodes  $$\al_j:= \begin{cases} \b_j,\  & j=0,1,\dots, m-1,\\
	1, \  & j=m.	
	\end{cases} $$ 
	It is easily seen that 	
	$$ \f 15 \f {n_1} {\sqrt{L}} <\f {n_1 \arccos (1-\f 1L)} {\pi} -1 \leq m <  \f {n_1 \arccos (1-\f 1L)} {\pi}<\f 3 5 \f  {n_1}{\sqrt{L}},  $$ 
	and \begin{equation}\label{7-11a}
	\al_j-\al_{j-1} \sim \f 1{n_1}\Bl( \sqrt{\al_j} +\f 1{n_1}\Br),\   \  j=1,2,\dots, m. 
	\end{equation}
	On the other hand, we partition the interval $[-b,b]$  with the uniform nodes: 
		 $$x_i=-b+\f {2i}{n_1} b,\   \ i=0,1,\dots, n_1.$$  
	Thus, we may define a partition of  $\overline{G}$ as follows:   for $1\leq i\leq n_1$ and $1\leq j\leq m$,
	\begin{align*}
		I_{i,j}:&=\Bl\{ (x,y)\in G:\  \   x_{i-1}\leq  x\leq x_i,\   \  \al_{j-1}\leq  g(x)-y \leq \al_{j}\Br\}.
	\end{align*}
 	For convenience, we also define, for each positive integer $\ell$, and $(i,j)\in\ZZ^2$, 
	\begin{align*}
		I^\ast_{i,j, \ell} :&= \Bl\{ (x,y)\in G:\  \   x_{i-1-\ell}\leq  x\leq x_{i+\ell},\   \  \al_{j-\ell-1}\leq g(x)-y \leq \al_{j+\ell}\Br\}.
	\end{align*}
	Here and elsewhere in the proof, we set
	$$ \al_j:=\begin{cases}
	0, &\  \  j<0\\
	\f 54, &\   j>m,
	\end{cases}\   \ \text{and}\   \
	x_i:= \begin{cases}
	-b, &\   \ i<0,\\
	b,&\  \ i\ge n_1.
	\end{cases}$$
%
%

Next,  for each $\xi=(\xi_x, \xi_y)\in G$, we denote by $\ta_\xi$ the angle in $[0,\pi]$ such that $g(\xi_x) -\xi_y =L-L\cos \ta_\xi$. Then,  for each $\xi=(\xi_x, \xi_y)\in G$,
\begin{equation*}
0\leq \ta_\xi =2 \arcsin  \sqrt{ \f{g(\xi_x)-\xi_y} {2L}}<\f \pi 5.
\end{equation*}
As a result,   for any $\xi=(\xi_x, \xi_y),\eta=(\eta_x, \eta_y)\in G$, 
	\begin{equation*}
		\Bl| \sqrt{g(\xi_x) -\xi_y} - \sqrt{g(\eta_x) -\eta_y}\Br|=2\sqrt{2L} \Bl|\sin \f {\ta_\xi-\ta_\eta}4\Br|\cos \f {\ta_\xi+\ta_\eta} 4\sim  |\ta_\xi-\ta_\eta|.
	\end{equation*}
	It follows  that 
	\begin{equation*}\label{3-2:bern}
	\wh \rho_G (\xi,\eta) \sim \max\Bl\{\|\xi_x-\eta_x\|,\  |\ta_\xi-\ta_\eta|\Br\},\  \ \forall \xi,\eta\in G.
	\end{equation*}
	Thus,  there exist two constants $c_1, c_2>0$ such that  
	\begin{equation}\label{3-5-0:bern}
		B_G(\og_{i,j}, c_1 n_1^{-1}) \subset I_{i,j}  \subset B_G(\og_{i,j}, c_2   n_1^{-1})
	\end{equation}	
	for  some  $\og_{i,j}\in I_{i,j}$ and all $1\leq i\leq n_1$ and $1\leq j\leq m$.  Clearly,  we may choose the parameter $\ell=\ell_\mu$  large enough so that 
	\begin{equation}\label{16-9}
		B_G(\og_{i,j}, c_2   n_1^{-1}) 	\subset B_G(\og_{i,j}, (3\mu +c_2)n_1^{-1}) \subset I_{i,j,\ell}^\ast.
	\end{equation}
Note that by   Fubini's theorem, 
\begin{equation}\label{3-6:bern}
|I_{i,j}|\sim n_1^{-2} (\sqrt{\al_j} +n_1^{-1}),\  \  1\leq i\leq n_1,\   \   1\leq j\leq m.
\end{equation}
	
Finally, we prove that if      $c_0: <\f 1 {4c_2}$, then  \eqref{3-4:berna} holds  for each $\da$-separated subset $\Ld$ of $G$. Indeed, since $n_1>\f 1{c_0 \da} $, the set  $\Ld\subset G$ is $\f 1{n_1 c_0}$-separated with respect to the metric $\wh{\rho}_G$, and since $2c_2<\f 1{c_0}$, \eqref{3-5-0:bern}    implies that   every $\og\in\Ld$ is contained in a unique set  $I_{i,j}$ with $1\leq i\leq n_1$ and $1\leq j\leq m$.  Thus, by~\eqref{3-5-0:bern}, \eqref{16-9} and the fact that $\f 1 {n_1} \sim \da$,
	it is enough to prove   that 
	\begin{equation}\label{3-2-1:bern}
		\Bl(\sum_{i=1}^{n_1} \sum_{j=1}^{m}  |I_{i,j}|  \bigl(\osc \bl(f;  I_{i,j,\ell}^\ast\br)\bigr)^p \Br)^{\f 1p}\leq C_{\ell} (n\da) \|f\|_{L^p(G^\ast)},\   \   \  f\in\Pi_n^2.
	\end{equation}

	To prove \eqref{3-2-1:bern}, we fix  $f\in\Pi_n^2$, and define $F(x,\al):= f(x, g(x)-\al)$ for $-2b\leq x\leq 2b$ and $\al\in\RR$. 
	Then $f(x,y)=F(x, g(x)-y)$, and   
	\begin{align*}
	\osc (f; I_{i,j,\ell^\ast}) &=\max_{ \sub{ x,x'\in [x_{i-1-\ell} , x_{i+\ell}]\\
			\al,\al' \in [\al_{j-\ell-1}, \al_{j+\ell}]}} |F(x,\al)-F(x',\al')|\\
		 &\leq 2 \sup_{\al \in [\al_{j-1-\ell}, \al_{j+\ell}]}  \sup_{x\in [x_{i-\ell-1}, x_{i+\ell}]} \Bl|F(x, \al) -\f{n_1}{2b}\int_{x_{i-1}}^{x_i}F(u, \al_j)\, du\Br|\\
		&	\leq 2 \Bl[ a_{i,j} (f) +b_{i,j}(f)\Br],
	\end{align*}
	where 
	\begin{align*}
		a_{i,j}(f):&=	  \sup_{\sub{x\in [x_{i-\ell-1}, x_{i+\ell}]\\
				\al \in [\al_{j-1-\ell}, \al_{j+\ell}]}} \Bl|F(x, \al) -\f{n_1}{2b} \int_{x_{i-1}}^{x_i}F(u, \al)\, du\Br|,\\
		b_{i,j}(f)&:=\f{n_1}{2b} \sup_{\al\in [\al_{j-1-\ell}, \al_{j+\ell}]}\Bl| \int_{x_{i-1}}^{x_i}[F(u, \al)-F(u, \al_j)]\, du\Br|.
	\end{align*}
	Thus, setting
		\begin{align*}
	\Sigma_1&:=\Bl(\sum_{i=1}^{n_1} \sum_{j=1}^{m}  |I_{i,j}| \bl| a_{i,j}(f)\br|^p\Br)^{\f 1p}\  \ \text{and}\  \  
	\Sigma_2:=\Bl( \sum_{i=1}^{n_1} \sum_{j=1}^{m} |I_{i,j}|   |b_{i,j} (f)|^p\Br)^{\f1p},
	\end{align*}
	we reduce to showing that 
	\begin{equation}
	\Sigma_k \leq C (n\da) \|f\|_{L^p(G^\ast)},\  \ k=1,2.
	 \label{3-7-0:bern}
	\end{equation}

To show \eqref{3-7-0:bern}, we first use  H\"older's inequality and \eqref{7-11a}  to obtain 
	\begin{align}
		|a_{i,j} (f)|^p&\leq  \f {C_\ell} {n_1^{p-1}}  
		\int_{x_{i-1-\ell}}^{x_{i+\ell}} \sup_{\al \in [\al_{j-1-\ell}, \al_{j+\ell}]}|\p_1 F(v, \al)|^p dv,\label{16-12}\\
		|b_{i,j}(f)|^p &\leq  C n_1(\al_{j+\ell} -\al_{j-1-\ell} )^{p-1} \int_{\al_{j-1-\ell}}^{ \al_{j+\ell}} \int_{x_{i-1}}^{x_i}|\p_2  F(u, \al)|^p\, du d\al\notag\\
		&\leq \f {C_\ell}  {n_1^{p-2}}( \f 1{n_1} +\sqrt{\al_j})^{p-1}
		\int_{\al_{j-1-\ell}}^{ \al_{j+\ell}} \int_{x_{i-1}}^{x_i}|\p_2  F(u, \al)|^p\, du d\al.\label{16-13}
	\end{align}

	Thus, using~\eqref{16-12} and~\eqref{3-6:bern}, we obtain 
	\begin{align*}
		\Sigma_1
		&\leq \f {C}{n_1}\Bl[  \sum_{i=1}^{n_1}  \int_{x_{i-1-\ell}}^{x_{i+\ell}} \Bl(\sum_{j=1}^{m} \f 1 {n_1} \Bl(\sqrt{\al_j} +\f 1 {n_1}\Br)
		\sup_{\al \in [\al_{j-1-\ell}, \al_{j+\ell}]} |\p_1 F(v, \al)|^p\Br)dv\Br]^{\f1p}.
	\end{align*}
	Since   the function $\p_1 F(v,\al)=(\p_1+g'(v)\p_2) f(v, g(v)-\al)$ is an algebraic polynomial of degree at most $n$ in the variable $\al$ for each fixed $v\in [-2b,2b]$, and since $n_1\ge n$,  it   follows from Lemma~\ref{lem-3-1:bern}  that 
	\begin{align*}
		\Sigma_1&\leq 	\f {C_\ell} {n_1} \Bl(  \sum_{i=1}^{n_1} \int_{x_{i-1-\ell}}^{x_{i+\ell}} \int_{0}^{\f 32 } |\p_1 F(v,\al)|^p\, d\al\, dv\Br)^{\f 1p}
		\leq C_\ell \da\Bl( \int_{-b}^b  \int_0^{\f 32} |\p_1 F(x,\al)|^p\, d\al\,  dx\Br)^{\f 1p}.\end{align*}
Thus, setting $\xi(v) =(1, g'(v))$ for $v\in [-2b, 2b]$, we obtain from Theorem~\ref{THM:2D BERN}   that  
	\begin{align*}
		\Sigma_1\leq C_\ell \da  \Bl( \int_{-b}^b \int_{0}^{\f 32} |\p_{\xi(x)} f(x, g(x)-\al)|^p\, d\al\, dx\Br)^{\f 1p} \leq C_{\ell} n\da \|f\|_{L^p(G^\ast)}, 
	\end{align*}
	which proves \eqref{3-7-0:bern} for $k=1$. 
	
	Similarly,   using~\eqref{16-13} and~\eqref{3-6:bern}, we have  
	\begin{align*}
		\Sigma_2
		&\leq  C_\ell \da  \Bl(\sum_{i=1}^{n_1} \sum_{j=1}^m  \int_{x_{i-1}}^{x_i}\int_{\al_{j-1-\ell}}^{\al_{j+\ell}}  |\p_2 F(u,\al)|^p \Bl(\sqrt{\al} +\f 1 {n_1}\Br)^p\, d\al\, du\Br)^{\f1p}\\
		&\leq  C\da\Bl( \int_{-b}^b  \int_0^{\f 54} |\p_2 F(u,\al)|^p (\sqrt{\al} )^p \, d\al\, du\Br)^{\f1p}+C\da n^{-1} \Bl( \int_{-b}^b  \int_0^{\f 54} |\p_2 F(u,\al)|^p  \, d\al\, du\Br)^{\f1p},\end{align*}
	which, using  	the univariate  Markov-Bernstein-type  inequality  (\cite[Theorem~7.3]{MT2} ),  is estimated above by 
	\begin{align*}
		&  C n\da  \Bl(\int_{-b}^b  \int_0^2 | F(u,\al)|^p \, d\al\, du \Br)^{\f 1p}\leq Cn\da \|f\|_{L^p(G^\ast)}.\label{3-9:bern}
	\end{align*}
	This proves \eqref{3-7-0:bern} for $k=2$, and hence completes the proof of \eqref{3-4:berna}.
\end{proof}

We also need a polynomial inequality on the unit ball  $\BB^d:=\{\xi\in\RR^{d}:\  \ \|\xi\|\leq 1\}$, which is stated in our third lemma.  Let $\rho_B$ denote the metric on $\BB^d$ given by  \begin{equation*}\label{metric:ball}\rho_B(\xi, \eta) =\|\xi-\eta\|+|\sqrt{1-\|\xi\|^2}-\sqrt{1-\|\eta\|^2}|,\  \ \xi,\eta\in \BB^d.\end{equation*}
For  $\xi\in \BB^d$ and $r>0$,  we  set 
$B_\rho(\xi, r):=\{\eta\in \BB^d:\  \  \rho_B(\xi, \eta) \leq r\}$. Given  $ B=B_\rho(\xi, r)$  and a constant $c>0$, we denote by  $cB$ the set $B_\rho (\xi, cr)$. It can be easily seen  that for any $r\in (0,1)$ and $\xi\in\BB^d$, 
\begin{equation}\label{Ap}
|B_\rho(\xi, r)|\sim |B_\rho(\xi, 2r)\setminus B_\rho(\xi, r)|\sim r^d \bigl( \sqrt{1-\|\xi\|^2} +r\bigr).
\end{equation}

\begin{lem}\label{lem-16-2-0}
	 Let 
	$\{\xi_j\}_{j=1}^m \subset \BB^d$ be 
	  $\f \da n$-separated with respect to the metric $\rho_B$ for some $\da\in (0,1)$ and $n\in\NN$, and let   $E := \bigcup_{j=1}^m B_\rho(\xi_j, \f \da {4n})$. If    $1\leq p< \infty$ and $0<\da<\da_0$,    where $\da_0\in (0,1)$ is a sufficiently small constant depending only on $d$,  then 
	$$\|f\|_{L^p(\BB^d)}\leq C \|f\|_{L^p(\BB^d\setminus E)},\   \   \  \forall f\in \Pi_n^{d},$$
	where the constant $C$ depends   only on $d$.\end{lem} 
\begin{proof} 
	According to  Theorem 11.6.1 of \cite[p.~290]{DX2}, there exists a constant $\da_0\in (0,1/4)$ depending only on $d$  such that for every $0<\da<\da_0$, every  maximal $\f {\da} n$-separated subset $\Ld$ of $\BB^d$ ( with respect to the metric $\rho_B$), and any $f\in\Pi_n^{d}$,
	\begin{align} C_1 &\Bl( \sum_{\xi\in\Ld} |B_\xi| \max _{\eta\in 4B_\xi} |f(\eta)|^p \Br)^{\f 1p} \leq \|f\|_{L^p(\BB^d)}
		\leq C_2 \Bl( \sum_{\xi\in\Ld} |B_\xi| \min _{\eta\in 4B_\xi} |f(\eta)|^p \Br)^{\f 1p},\label{16-5}\end{align}
	where $B_\xi=B_\rho(\xi, \f {\da} n)$ for $\xi\in\Ld$,  and the constants $C_1, C_2>0$ depend only on $d$.
  Now let
$\{\xi_j\}_{j=1}^m \subset \BB^d$ be 
$\f \da n$-separated with respect to the metric $\rho_B$ with $0<\da<\da_0$, and  take  a  maximal $\f {\da} n$-separated subset   $\Ld\subset \Og$   which contains  all the points $\xi_j$, $j=1,\dots, m$.  Then  using~\eqref{16-5} and \eqref{Ap},  we have 
	\begin{align*}
		\|f\|_{L^p(\BB^d)}
		&\leq  C_2 \Bl( \sum_{\xi\in\Ld}  \f {|B_\xi|} {  |(2^{-1}B_\xi)\setminus (4^{-1} B_\xi)|} \int_{(2^{-1}B_\xi)\setminus (4^{-1}B_\xi)} |f(\eta)|^p\, d\eta  \Br)^{\f 1p}\\
		&\leq C \Bl( \sum_{\xi\in\Ld} \int_{(2^{-1}B_\xi)\setminus (4^{-1}B_\xi)}
		 |f(\eta)|^p\, d\eta  \Br)^{\f 1p}\leq C \|f\|_{L^p(\BB^d\setminus E)},
	\end{align*}
	where the last step uses the facts that the sets $2^{-1} B_\xi$, $\xi\in\Ld$ are pairwise disjoint and 
	$E\subset \bigcup_{\xi\in\Ld} 4^{-1} B_\xi$. 
		 This completes the proof of Lemma \ref{lem-16-2-0}. 
\end{proof}

Finally, we need a lemma  to deal with points in $\Og$ that are not very close to the boundary $\Ga$. 
Given a parameter $\da\in (0,1)$, we set    $\Ga_\da:=\{ \xi\in\Og:\  \  \dist(\xi,\Ga) \leq  \da\}$ and $\Og_\da=\Og\setminus \Ga_\da$. Also, recall  that 
$$ B_{r} (\xi) =\Bl\{\eta\in\RR^{d}:\  \  \|\eta-\xi\|< r\Br\},\  \ r>0,\   \ \xi\in\RR^d.$$

\begin{lem}\label{lem-16-2b}  Let $\mu>1$ and $\da_0\in (0,1)$  be two given  parameters. Let    $\va\in (0,\f {\da_0}{8})$, and let $n$ be an integer $>\mu$.  Assume that
	$\Ld$ is a finite subset of $\Og_{\da_0}$ satisfying that $\|\xi-\eta\|\ge \f \va n$ for any two distinct points $\xi, \eta\in\Ld$.  	
	Then any  $f\in\Pi_{n}^{d}$ and $1\leq p< \infty$,   
	\begin{equation}\label{16-17} \Bl(\Bl(\f \va n\Br)^{d} \sum_{\xi\in\Ld} \bigl(\osc \bl(f; B_{ {\mu\va} /n} (\xi)\br)\bigr)  ^p\Br)^{\f1p} \leq C_{\mu,\da_0}\,\va \|f\|_{L^p(\Og)},\end{equation}
	where 	the constant $C_{\mu, \da_0}$ is independent of $\va$, $n$ and $f$.

\end{lem}

\begin{proof} 	We first  cover $\Og_{\da_0}$ with finitely many closed Euclidean  balls $B_1,\dots, B_{n_0}\subset \Og$ of radius ${\da_0}/8$ such that $6B_i \subset\Og$ for all  $1\leq i\leq n_0$, where $n_0$ depends only on ${\da_0}$ and $\Og$.  We then reduce  to showing that for  $i=1,\dots, n_0$,
	\begin{equation}\label{16-18} \Bl( \Bl(\f \va n\Br)^{d} \sum_{\xi\in\Ld\cap B_i} \Bl|\osc \bl(f;  B_{{\mu\va}/ n} (\xi)\br)\Br|^p\Br)^{\f1p} \leq C_{\mu,{\da_0}}\,\va \|f\|_{L^p(\Og)},\   \ \forall f\in\Pi_{n}^d.\end{equation}
	
	To show \eqref{16-18}, for  each $\xi\in \Ld$, we set 	$f_{\xi}:=\f 1{|B_{{\mu\va}/ n} (\xi)|} \int_{B_{{\mu\va}/ n} (\xi)} f(\eta)\, d\eta$, and let 
	$\eta_\xi\in  B_{{\mu\va}/ n} [\xi]$ be such that 
	$$ \max_{ \eta\in  B_{{\mu\va}/ n} [\xi]} |f(\eta)-f_{ \xi} |=|f(\eta_\xi)-f_{\xi} |.$$
	Using  Poincare's inequality, we obtain  that 
	for  any   $\xi\in\Ld\cap B_j$,
	\begin{align*}
		\osc \bl(f; B_{{\mu\va}/ n} (\xi)\br)& \leq 2 |f(\eta_\xi)-f_{ \xi}| \leq C_d \int_{B_{{\mu\va}/ n} (\xi)}
		\|\nabla f(\eta)\|\|\eta-\eta_\xi\|^{-d+1} \, d\eta\\
		&\leq C_\mu \int_{2 B_j}
		\|\nabla f(\eta)\|\|\eta-\eta_\xi\|^{-d+1} \Bl(1+ \f n \va\| \eta-\eta_\xi\|\Br)^{-d-1}\, d\eta,
	\end{align*}
where the last step uses the fact that $B_{{\mu\va}/ n} (\xi) \subset 2 B_j$ for any $\xi\in \Ld\cap B_j$.	
	By Lemma~\ref{lem-16-2-0}, this implies that for each $\xi\in\Ld\cap B_j$,
	\begin{align*}
		\osc \bl(f; B_{{\mu\va}/ n} (\xi)\br)	&\leq  C \int_{\{\eta\in 2 B_j:\  \  \|\eta-\eta_\xi\| \ge \f {\va} {4n}\}}
		\|\nabla f(\eta)\|\|\eta-\eta_\xi\|^{-d+1} \Bl(1+ \f n\va\| \eta-\eta_\xi\|\Br)^{-d-1}\, d\eta\\
		&\leq C \Bl(\f n\va\Br)^{d-1}\int_{\Og_{{\da_0}/4}}
		\|\nabla f(\eta)\|\Bl(1+ \f n\va \| \eta-\eta_\xi\|\Br)^{-2d}\, d\eta.
	\end{align*}
Thus, using  H\"older's inequality, we obtain  that for any $\xi\in\Ld\cap B_j$,
	\begin{align*}
		\Bl| \osc \bl(f;  B_{{\mu\va}/ n} (\xi)\br)\Br|^p &\leq C  \Bl( \f n\va\Br) ^{d-p}\int_{\Og_{{\da_0}/4}}
		\|\nabla f(\eta)\|^p\Bl (1+ \f n\va\| \eta-\eta_\xi\|\Br)^{-2d}\, d\eta.
	\end{align*}
	It follows that   \begin{align}
		\Bl(\f \va n\Br)^{d} &\sum_{\xi\in \Ld\cap B_j} \Bl|\osc \bl(f;  B_{{\mu\va}/ n} (\xi)\br)\Br|^p\notag\\
		&\leq C  \Bl(\f \va n\Br)^p  \int_{\Og_{{\da_0}/4}}
		\|\nabla f(\eta)\|^p \sum_{\xi\in\Ld\cap B_j} \Bl(1+ \f n\va\| \eta-\eta_\xi\|\Bl)^{-2d}\, d\eta. \label{16-21}\end{align}
However, 	since $\Ld\subset \Og_{\da_0}$ is   $\f \va  {n}$-separated  in $\RR^d$,  we obtain that for any $\eta\in\RR^{d}$, 
	\begin{align*}
		\sum_{\xi\in\Ld\cap B_j} &\Bl(1+ \f n\va\| \eta-\eta_\xi\|\Br)^{-2d}\leq C 	 \Bl( \f n\va\Br)^d \sum_{\xi\in\Ld\cap B_j} \int_{B_{\va/n} (\xi)} \Bl (1+ \f n\va\| \eta-\zeta\|\Br)^{-2d}\, d\zeta\\
		&\leq C 	 \Bl( \f n\va\Br)^d \int_{\RR^{d}}  \Bl(1+ \f n\va\| \eta-\zeta\|\Br)^{-2d}\, d\zeta
		\leq  C_d <\infty,
	\end{align*}
	where the first step uses the fact that   $\eta_\xi\in  B_{\f {\mu\va}{n}}(\xi)$ for each $\xi\in \Ld\cap B_j$.
	This combined with~\eqref{16-21} yields  \begin{align*}
		\Bl( (\f \va n)^{d} \sum_{\xi\in\Ld\cap B_j} \Bl|\osc \bl(f;  B_{{\mu\va}/ n} (\xi)\br)\Br|^p\Br)^{\f1p}&\leq \f {C\va}n \Bl(\int_{\Og_{{\da_0}/4}}
		\|\nabla f(\eta)\|^p \, d\eta\Br)^{\f1p}\\
		&\leq  C\va \|f\|_{L^p(\Og)},\end{align*}
	where the last step uses the regular  Bernstein inequality.
	This completes the proof of Lemma~\ref{lem-16-2b}.
\end{proof}

We are now in a position to prove Theorem~\ref{thm-16-1:MZ}.\\

\begin{proof}[Proof of Theorem~\ref{thm-16-1:MZ}]  Without loss of generality, we may assume that $n\ge N_\Og$, where $N_\Og$ is a sufficiently large positive integer depending only on $\Og$. 
	Indeed, \eqref{16-17} for $n\leq N_\Og$ can be deduced directly  from the case $n=N_{\Og}$.
	
	By Proposition ~\ref{lem-2-1-18} and Remark~ \ref{rem-2-7},  we can find a constant $\ld_0\in (0,1)$ (as small as we wish)  and finitely many domains $G_1,\dots, G_{m_0}\subset \Og$ of special type attached to $\Ga$ such that 
	$$ \Ga_{\da_0} :=\{\xi\in\Og:\  \ \dist(\xi, \Ga) \leq {\da_0} \} \subset \bigcup_{j=1}^{m_0} G_j(\ld_0)$$
	  with $m_0\in\NN$ and  $\da_0\in (0,1)$ depending  only on $\Og$.  Then $\Ld=\bigcup_{j=1}^{m_0+1}\Ld_j$, where 
	$ \Ld_{j} =\Ld \cap G_j(\ld_0)$ for $j=1,\dots, m_0,$ 
	$\Ld_{m_0+1}  =\Ld\cap \Og_{\da_0}$ and $\Og_{\da_0}= \Og\setminus \Ga_{\da_0}$.  Thus,
	\begin{align*}
		&\sum_{\xi\in\Ld} \Bl| U(\xi, \f \va n) \Br| \Bl|\osc (f; U (\xi, \f {\ell \va} n)) \Br|^p  \leq \sum_{j=1}^{m_0+1} S_j,\end{align*}
	where 
	\begin{align*}S_j:=\sum_{\xi\in\Ld_j} \Bl| U(\xi, \f \va n) \Br| \Bl|\osc (f; U (\xi, \f {\ell \va} n)) \Br|^p,\  \  \ 1\leq j\leq m_0+1.
	\end{align*}
	Clearly, by~\eqref{metric}  and~\eqref{rhog-0}, we  may choose $N_\Og$ sufficiently large so that 
	$$ B_{G_j} (\xi,  \f {\mu_1\va}n)\subset U(\xi, \f {\ell\va}n) \subset B_{G_j} (\xi,  \f {\mu_2\va}n)\subset G_j,\  \ \forall \xi\in\Ld_j,\  \ 1\leq j\leq m_0,$$
	and 
	$$B_{\f {\mu_1\va} n} (\xi)\subset  U(\xi, \f {\ell\va}n) \subset  B_{\f {\mu_2 \va} n} (\xi) \subset \Og_{{\da_0}/2},\  \  \forall \xi\in \Ld_{m_0+1},$$
	where $\mu_2>1>\mu_1>0$, and $\mu_1, \mu_2$ depend only on $\ell$ and $\Og$.

	For $1\leq j\leq m_0$,   we use~\eqref{rhog-0} and  Lemma~\ref{lem-16-2} to obtain  	
	\begin{align*}S_j&\leq C
		\sum_{\xi\in\Ld_{j}} \Bl| B_{G_j} (\xi, \f { \mu_2\va} n) \Br| \Bl|\osc (f; B_{G_j}  (\xi, \f {\mu_2\va} n)) \Br|^p	\leq \Bl(C \va \|f\|_{L^p(\Og)}\Br)^p.
	\end{align*}
	
	For $j=m_0+1$, we use~\eqref{metric}  and Lemma~\ref{lem-16-2b} to obtain 
	\begin{align*}	
		S_{m_0+1}:&\leq C\sum_{\xi\in\Ld_{m_0+1} } \Bl(\f \va n\Br)^{d}  \Bl|\osc (f; B_{\f {\mu_2 \va} n}(\xi) \Br|^p 
		\leq \Bl(C \va \|f\|_{L^p(\Og)}\Br)^p.
	\end{align*}
	
	Putting the above together, we prove the desired estimate~\eqref{3-4:bern}, and therefore,  complete the proof of Theorem~\ref{thm-16-1:MZ}.
\end{proof}
\section{Acknowlegement}

The authors would like to thank the referees very much for their valuable comments which helped to improve the manuscript.

%

\end{document}